\documentclass[11pt, twoside, a4paper]{article}
\usepackage[english]{babel}
\usepackage{amsmath,amssymb,amsthm,epsfig}
\usepackage{graphicx}
\usepackage[export]{adjustbox}
\begin{document}

\def\appendixname{\empty}

\def\refname{References}
\def\bibname{References}
\def\chaptername{\empty}
\def\figurename{Fig.}
\def\abstractname{Abstract}

\title{  The Dual Potential, the involution kernel and Transport in Ergodic Optimization}
\author{
Artur O. Lopes \,\,\footnote{Instituto de Matem\'atica, UFRGS, 91509-900 Porto Alegre, Brasil. Partially supported by CNPq.}, Elismar O. Oliveira \,\, \footnote{Instituto de Matem\'atica, UFRGS, 91509-900 Porto Alegre, Brasil.}\,\,\,and\,\,
Phillippe Thieullen\,\,\footnote{Institut de Math\'ematiques, Universit\'e Bordeaux, F-33405 Talence, France.}}
\date{\today}
\maketitle

\begin{abstract}

Consider the shift $\sigma$ acting on the Bernoulli
space $\Sigma=\{1,2,...,n\}^\mathbb{N}$. We denote $\hat{\Sigma}=
\{1,2,...,n\}^\mathbb{Z}=\Sigma\times \Sigma$. We analyze several properties of the
maximizing probability $\mu_{\infty,A}$ of a Holder potential $A:
\Sigma \to \mathbb{R}$. Associated to $A(x)$, via the involution kernel, $W(x,y)$,  $W: \hat{\Sigma} \to \mathbb{R}$,
one can  get the dual potential $A^*(y)$, where
$(x,y)\in \hat{\Sigma}$. We denote  $\mu_{\infty, A^*}$ the
maximizing probability for $A^*$.

We would like to consider the transport problem from
$\mu_{\infty,A}$ to  $\mu_{\infty,A^*}$. In this case, it is
natural to consider the cost function $c(x,y) = I(x) - W(x,y)
+\gamma $, where $I$ is the deviation function for
$\mu_{\infty,A}$, as the limit of Gibbs probabilities $\mu_{\beta
A}$ for the potential $\beta A$ when $\beta \to \infty$. The value
$\gamma $ is a constant which depends on $A$. We could also take $c=-W$ above. We denote by ${\cal K}={\cal K}( \mu_{\infty,A}, \mu_{\infty,
A^*})$ the set of probabilities $\hat{\eta} (x,y)$ on
$\hat{\Sigma}$, such that $ \pi_ x^* (\hat{\eta} ) =
\mu_{\infty,A}, \, \, \text{and}\,\,\pi_ y^* (\hat{\eta} ) =
\mu_{\infty,A^*} \,.$

We describe the minimal solution $\hat{\mu  }$  (which is invariant by the shift on $\hat{\Sigma}$) of  the
Transport Problem, that is, the solution of
$$
\inf_{\hat{\eta} \in {\cal K}  } \int \int c(x,y) \,
d\,\hat{\eta} =  \,-\,  \max_{\hat{\eta} \in {\cal K}  } \int \int
(W(x,y) - \gamma ) \, d\,\hat{\eta}. \,
$$
The optimal pair of functions for the Kantorovich Transport dual Problem
is  $(-V,-V^*$), where we denote the two calibrated sub-actions by $V$ and $V^*$, respectively,
for $A$ and $A^*$. We show that
the involution kernel $W$  is cyclically monotone. In other words, satisfies a twist condition in the support of $\hat{\mu  }$ We analyze the question:
is the support of $\hat{\mu  }$ a graph?
We also investigate the question of finding an explicit expression for
the function $f:\Sigma \to \mathbb{R}$ whose $c-$subderivative
determines the graph.

We also analyze  the same kind of problem for expanding transformations on the circle.

\end{abstract}


\newpage

\vspace {.8cm}

\theoremstyle{definition}

\newtheorem{exemp}{Example}

\newtheorem{theorem}{Theorem} \newtheorem{corollary}{Corollary}
\newtheorem{lemma}{Lemma} \newtheorem{proposition}{Proposition}
\newtheorem{definition}{Definition}

\newcommand{\z}{\mathbb{Z}} \newcommand{\re}{\mathbb{R}}
\newcommand{\tn}{\mathbb{T}^N}
\newcommand{\rn}{\mathbb{R}^N}
\newcommand{\espaco}{[0,1]^{\mathbb{N}}}
\newcommand{\supp}{\mbox{supp}}
\newcommand{\nat}{\mathbb{N}}
\newcommand{\Rm}{{\noindent \sc Note. \ }}

\def\cqd {\,  \begin{footnotesize}$\square$\end{footnotesize}}
\def\cqdt {\hspace{5.6in} \begin{footnotesize}$\blacktriangleleft$\end{footnotesize}}


\bigskip
\section{Introduction }\label{secaoinicial}

It seems  natural to try to investigate the connections of Transport Theory with Ergodic Theory. Some results on this direction appear in \cite{Kl}, \cite{KLS}, \cite{KGLM}, \cite{BOR}, \cite{Sou}
and \cite{GP}. Here we follow a different path.

Given a continuous function
$A:\Sigma= \{1, 2, 3,.., d\}^\mathbb{N}  \to \mathbb{R}$, we call $\mu_{\infty,A}$ a maximizing probability for $A$,
if $\int A d\nu$ attains the maximal value in $ \mu_{\infty,A}$, when the  probabilities $\nu$ range among the set of  invariant for the shift
acting on the Bernoulli space $\Sigma$. We denote by $m(A)$ this maximal value.

Such maximizing probabilities $\mu_{\infty,A}$ can be seen as  the equilibrium states at zero temperature  for a system on the one dimensional lattice $\mathbb{N}$ with $d$ spins in each site and under the influence of an interacting potential $A$ (see \cite{BLLco}, \cite{CG} \cite{CLT} \cite{Le} \cite{J1} \cite{B1} \cite{Mo} and \cite{LMMS1}).

A main conjecture on the area claims that for a generic Holder potential $A$ the maximizing probability has support in a unique periodic orbit for the shift (for a partial result see \cite{CLT}).
This conjecture was recently proved by G. Contreras (see \cite{CO}).

We address the question of finding the optimal transport plan from a certain maximizing probability
to another. More precisely, we would like to consider the transport problem from
$\mu_{\infty,A}$ to  $\mu_{\infty,A^*}$, where $A:\Sigma= \{1, 2, 3,.., d\}^\mathbb{N}  \to \mathbb{R}$ is a Holder potential and $A^*$ its dual (see \cite{BLT}).

We consider here that $A$ acts on the variable $x$ and $A^*$ in the variable $y$. A function $W(x,y)$ called the involution kernel will play an important role in the theory. The twist condition for $W$ is a kind of convexity assumption. We will describe bellow with all details the setting we are going to consider in the present paper. We will also provide several examples to illustrate the theory.

We assume here in most (but not all) of the results that the maximizing probability $\mu_{\infty,A}$ (on $\Sigma$) for $A$ is unique.

We denote by $\hat{\mu  }$  the minimizing probability over
$$\hat{\Sigma}=\{1, 2, 3,.., d\}^\mathbb{Z}=\Sigma \times \Sigma,$$
for the natural Kantorovich
Transport Problem associated to the $-W$, where $W(x,y)$, for  $(x,y) \in \Sigma \times \Sigma,$ is the involution kernel associated to $A$ (see \cite{BLT}).

We will denote by $\hat{\sigma}$ the shift on $\hat{\Sigma}$.
The probability $\hat{\mu  }_{max}$,  the natural extension of $ \mu_{\infty,A}$, is described in \cite{BLT}.

We point out that by its very nature the Classical Transport Theory is not a Dynamical Theory (in the sense of considering invariant probabilities) \cite{Vi1} \cite{Vi2} \cite{Ra}. One has to consider a cost which is obtained from dynamical properties in order to get
optimal plans which are invariant for $\hat{\sigma}$.

Recent results in Ergodic Transport are \cite{LM4}, \cite{GL4}, \cite{CLO}, \cite{LMMS}, \cite{OM} and \cite{LO}.

\medskip

We will consider a cost which is the involution kernel $W$.
\medskip

First we show that:

\begin{theorem} The minimizing Kantorovich
probability $\hat{\mu}$ on $\hat{\Sigma}$  associated to $-W$, where $W$ is the involution kernel for $A$, is
$\hat{\mu  }_{max}$.
\end{theorem}
\medskip

One of our main results is Theorem \ref{mai} which claims that the support of
$\hat{\mu}_{max} $ is $W$-cyclically monotone. We do not assume the twist condition in the above result

\medskip

The calibrated subactions $V$ play an important role in Ergodic Optimization. They can help to find the support of the maximizing
probability (see \cite{BLLco}, \cite{J1} or \cite{CLT} for instance). Moreover, if we denote  $R(x) = V(\sigma(x))- V(x) -A(x) + m(A)$, then
$ I(x) = \sum_n R(\sigma^n (x)))$ defines a nonnegative lower semicontinuous function (can be infinite at several points) which is the deviation function for the family
of Gibbs states associated to $A$ when the temperature  converges to zero \cite{BLT} (see \cite{BCLMS} \cite{LM4} for the case of the $XY$ model). For a class of explicit nontrivial examples of subactions $V$ see \cite{BLM}.

\begin{theorem} If $V$ is the calibrated subaction for $A$, and  $V^*$ is the calibrated  subaction for $A^*$, then,
the pair $(-V,-V^*)$ is the dual ($-W+I$)-Kantorovich  pair of $(\mu_{\infty,A},\mu_{\infty,A^*})$, when $I$ is the deviation function for $A$.
\end{theorem}

Finding the optimal transport measure between two probabilities is the solution of the so called relaxed problem \cite{Vi1}.
If we want to find a measurable transformation (the Monge problem) which  transfers one probability to another we need to show that the
graph property is true in the support of such probability (which does not always happen if one considers a general cost function) \cite{Vi1}.

Finally, we analyze here the graph property for the support of the $\hat{\mu  }_{max}$  (over  $\hat{\Sigma}=\{1, 2, 3,..,d\}^\mathbb{Z}$) which is  the minimizing probability
for the cost function $-W$.

One can consider in the Bernoulli space
$\Sigma=\{0,1\}^\mathbb{N}$ the lexicographic order. In this way,
$x<z$, if and only if, the first element $i$ such that, $x_j=z_j$
for all $j<i$, and $x_i \neq z_i$, satisfies the property $x_i<
z_i$. Moreover, $(0,x_1,x_2,...)<   (1,x_1,x_2,...).$

One can also consider the more general case
$\Sigma=\{0,1,...,d-1\}^\mathbb{N}$, but in order to simplify the
notation  and to avoid technicalities, we consider only the case
$\Sigma=\{0,1\}^\mathbb{N}$.

\begin{definition} We say a  continuous  $G: \hat{\Sigma}=\Sigma\times \Sigma \to \mathbb{R}$ satisfies the twist condition on
$ \hat{\Sigma}$, if  for any $(a,b)\in \hat{\Sigma}=\Sigma\times
\Sigma $ and $(a',b')\in\Sigma\times \Sigma $, with $a'> a$,
$b'>b$, we have
\begin{equation}
G(a,b) + G(a',b')  <  G(a,b') + G(a',b).
\end{equation}
\end{definition}

The twist condition is inspired in the Aubry-Mather Theory \cite{Ban} \cite{CI} \cite{Go} \cite{GT1} \cite{GT2}. It is a quite natural concept in Classical Optimization and Transport Theory \cite{Mi} \cite{Ba} \cite{Del} \cite{Vi1} \cite{Vi2} \cite{Ra} \cite{CLO} \cite{LOS} (see \cite{LO} for dynamical examples).

The twist condition is also described by the concept of {\bf global} cyclically monotonicity (see \cite{Vi1})

We point out that in Mather Theory in order to have the graph property (see \cite{Mat}  \cite{CI}) for the minimal action measure it is necessary to assume that that Lagrangian is convex in the velocity. We need in our setting some technical assumptions to replace this important property. We believe that the twist condition is the natural one.

\begin{definition} We say a continuous $A: \Sigma \to \mathbb{R}$   satisfies the twist condition, if its involution kernel $W$
satisfies the twist condition.
\end{definition}

The involution kernel  of $A$ is not unique (see \cite{BLT}), but if the above property is true for some $W$, then it will also be true for any other one.

Our final result is:

\begin{theorem} Suppose $W$ satisfies the twist condition on $\hat{\Sigma}$, then, the support of $\hat{\mu}_{max}= \hat{\mu}$ on $\hat{\Sigma}$ is a graph.

\end{theorem}

We point out that it can exist (not always) a single point in the support of $\hat{\mu}$ such that its orbit  has two points in the support of the vertical fiber. But this orbit  is a zero measure set.

A similar definition can be consider for an expanding transformation on $[0,1]$, and we are also able to get the analogous   graph property result. This also includes the case of $T(x)=-\, 2 x$ (mod 1).

We present in  the appendix at the end of the paper several examples (and computations) where one can write the involution kernel $W$  explicitly and  the twist condition is satisfied.

First we will explain all the preliminaries we will need later.
\vspace{0.2cm}

Consider $ X $ a compact metric space. Given a continuous
transformation $ f: X \to X $, we denote by $ \mathcal M_f $ the
convex set of $f$-invariant Borel probability measures. As usual, we
consider in  $ \mathcal M_f $ the weak* topology.

The standard model used in ergodic optimization is the triple $ (X,
f, \mathcal M_f) $. Given a potential $ A \in C^0(X) $,
we denote
\begin{equation}
m(A)=\max_{\nu \in \mathcal M_f} \int_X A(x) \;
d\nu(x).
\end{equation}

We are interested here in  the characterization and main properties of $A$-maximizing
probabilities, that is, the probabilities belonging to the set

\begin{equation}
\{\mu \in \mathcal M_f: \int_X A(x) \; d\mu(x) = m(A) \}.
\end{equation}

We will assume here that $A$ is Holder.

In the following we will also assume that the maximizing probability
$\mu_{\infty,A}=\mu_\infty$ is unique.

Under reasonable hypothesis (expanding, hyperbolic, etc.)
several results were obtained related to this maximizing question,
among them \cite{BLLco, CG,BLT, B1, B2, CLT, HY, J1, Mo, Le,
J2, LT1, TZ, Sa, BG, GT1, GT2}. For maximization with constraints see \cite{GL1,
LT2}. Questions related to the dynamics on the boundary of the fat attractor appear in \cite{LO}. Naturally, if we change the maximizing notion for the
minimizing one, the analogous properties will also be true.

Our focus here will be mainly on symbolic dynamics and on
expanding transformations on $S^1$ or the interval $[0,1]$. We recall some
basic definitions (see \cite{BLLco} or \cite{CLT} for example).

So let $ \sigma: \Sigma \to \Sigma $ be a subshift of finite type
defined by a matrix $C$ of $0$ and $1$, where $\sigma (x_0,x_1,x_2,..)=  (x_1,x_2,x_3,..).$ In this case we are considering $X=\Sigma= \{1,
2, 3,.., d\}^\mathbb{N}_C$  and $f=\sigma$. Remind that, for a fixed $ \lambda \in (0, 1)
$, we consider for $ \Sigma $ the metric $ d (\mathbf x,
\bar{\mathbf x}) = \lambda^k $, where $ \mathbf x = (x_0, x_1,
\ldots), \bar{\mathbf x} = (\bar x_0, \bar x_1, \ldots) \in \Sigma $
and $ k = \min \{j: x_j \ne \bar x_j \} $. In this situation, given
a H\"older potential $ A:
 \{1, 2, 3,.., d\}^\mathbb{N} \to \mathbb{R} $, one should be interested in $A$-maximizing
probabilities for the triple $ (\Sigma, \sigma, \mathcal M_\sigma)
$,  where the probabilities are consider  over
${\cal B}$, the $\sigma$-algebra of Borel of $\Sigma$. In order to simplify the notation here we will consider the
full  Bernoulli space (all entries of $C$ are equal to $1$).

Given an $C^{1+\alpha}$ expanding transformation $T$ of fixed degree on $S^1$ and $A:S^1 \to \mathbb{R}$ we will be
interested in $A$- maximizing probabilities on $(S^1,T , \mathcal M_T),$ where the probabilities are consider  over
${\cal B}$, the $\sigma$-algebra of Borel of $S^1$.

One can consider the analogous setting for $C^{1+\alpha}$  expanding transformations of fixed degree over $[0,1]$.

Convex potentials $A:[0,1]\to \mathbb{R}$ and the transformation $T: [0,1]\to [0,1]$, given by
$T(x)=2\, x$ (mod $1$),  were considered in
\cite{J3} where it was shown that the maximizing
probabilities in this case are Sturm measures. For $T(x)$ equal to
$-\,2\,x$ (mod 1) however, the situation is completely different (see \cite{JS}).

\begin{definition}
A function $ u \in C^0(\Sigma) $ is a sub-action for the potential $
A $ if, for any $ \mathbf x \in \Sigma=  \{1, 2, 3,..,
d\}^\mathbb{N}_C $, we have

\begin{equation}
u(\mathbf x) \le u(\sigma(\mathbf x) ) - A(\mathbf x) + \beta_A .
\end{equation}
\end{definition}

Let $ (\Sigma^*, \sigma^*) $ be the dual subshift.

In the case of the full Bernoulli space (all entries of $C$ equal $1$) then
$ \Sigma^*= \{1, 2, 3,..,
d\}^\mathbb{N}  $ and
$\sigma^*  (y_0,y_1,y_2,..)=  (y_1,y_2,..).$

We consider the
space of the dynamics $ (\hat \Sigma, \hat \sigma) $, the natural
extension of $ (\Sigma, \sigma) $, as subset of $ \Sigma^* \times
\Sigma $. In fact, if $ \mathbf y = (\ldots, y_1, y_0) \in \Sigma^*
$ and $ \mathbf x = (x_0, x_1, \ldots) \in \Sigma $, then $ \hat
\Sigma $ will be the set of points
$$ < y, x> =
(\ldots, y_1, y_0 | x_0, x_1, \ldots) \in \Sigma^* \times \Sigma ,$$
such that $ (y_0, x_0) $ is an allowed word (no restrictions when we consider the full Bernoulii space). In this case
$$\hat{\sigma} \, (\ldots, y_1, y_0 | x_0, x_1, \ldots)\,=\, (\ldots, y_1, y_0, x_0 | x_1, x_2, \ldots) .$$

We point out that we use here the notation $<y,x>=(x,y)$. For functions $b: \hat{\Sigma}\to \mathbb{R}$, we denote its value on $<y,x>$ by $b(x,y)$.

We define the map $ \tau: \hat \Sigma \to \Sigma $ by $ \tau(x, y) = \tau_{\mathbf y}(\mathbf x) = (y_0, x_0, x_1,
\ldots) $.

Note that, if $\pi_x: \hat{\Sigma}\to \Sigma$ is the projection in
the $x$ coordinate, then, $\tau_y(x)= \pi_x \circ \hat{\sigma}^{-1}\,(x,y). $

We denote by $\pi_y(x,y)=y$ the projection on the second coordinate.

 Note
that $\hat{\sigma}^{-1}(x,y) = (\tau_y (x),\sigma^* (y) ).$
\vspace{0.2cm}

\begin{definition}
A continuous function $V:\Sigma\to\mathbb{R}$ is
called calibrated subaction for $A$, if
\[
V(x)=\max_{z\,:\,\sigma(z)=x}\big(V(z)+A(z)-m(A)\big).
\]
(In other terms, $V$ is a calibrated subaction if
for any $x\in \Sigma$, there
exists $z \in\Sigma$, such that, $\sigma  (z) = x$, and
$V(z)+A(z)-m(A)=  V(x)$\,).
\end{definition}

Note that  for all $z$ we have
$$  V(\sigma(z))- V(z) -A(z) + m(A)\geq 0.$$

We show  bellow some explicit expressions for calibrated subactions for a class of  potentials $A$.

We point out that  we will also consider here analogous results for an expanding transformation $T:S^1\to S^1$
(or, $T:[0,1]\to[0,1]$) of class
$C^{1 + \alpha}$, and a Holder potential $A:S^1\to \mathbb{R}$ (or, $A:[0,1] \to \mathbb{R}$) as in \cite{CLT}.
The case $T(x)=-\, 2 x $ (mod 1) is one of the examples we have on mind.
\vspace{0.2cm}

In this case one could consider analogous problems in $S^1 \times S^1$, or, $S^1 \times \Sigma$, if one consider
the symbols $i$ which index the inverse branches $\tau_i$  of $T$ \cite{LOS} \cite{LO}. The existence of involution kernel, L.D.P. properties,
etc, are also true.

The calibrated sub-action is unique (up to an additive constant) if the maximizing probability is unique (see \cite{CLT} \cite{BLT} \cite{GLT}).

We point out that we called strict
in \cite{BLT} what we denote here by calibrated.

We will use from now on the notation  of \cite{BLT}.

\begin{definition}
Given $A:\Sigma \to \mathbb{R}$ Lipchitz,  consider $A^* (y)$    (the dual
potential), where $A:\Sigma^* \to \mathbb{R}$, and $W(x,y)=W_A
(x,y)$ its involution kernel.

This means, by definition that for all $<y,x>=(x,y)\in \hat{\Sigma}$

\begin{equation}
A^* (y) = A ( \tau_y (x) ) + W ( \tau_y (x),\sigma^* (y)  ) - W(x,y).
\end{equation}
\end{definition}

This expression can be also written in the form

$$A^* (x,y) = A( \hat{\sigma}^{-1}(x,y)) + W ( \hat{\sigma}^{-1}(x,y)) - W(x,y).$$

If  $A$ depends on just two coordinates we can take $A^*$ as the transpose of $A$. Therefore, the above definition extends this concept in the case $A$ depends on infinite coordinates on the Bernoulli space.
We say $A$ is involutive if $A=A^*$.
\medskip

We address the question of regularity of the involution kernel $W$ (is bi-Holder) in the item d) in the Appendix.
\medskip

We denote by $M$ the Bernoulli space or unitary circle.

Suppose $T$ is an expanding transformation on $M$ ($T$ can be the shift $\sigma$ or the transformation $T$ defined above).

For a Lipchitz potential $A:M \to \mathbb{R}$  the pressure of $A$ is the value
$$ P(A)= \sup_{\mu\,\,\text{invariant for}\, T} \, \{h(\mu) + \int A \, d\mu\,\},$$
where $h(\mu)$ is the Kolmogorov entropy of the invariant probability $\mu$.

The equilibrium state for $A$ is the probability $\mu$ which realizes the above aupremum.

Given a Holder function $A: M \to \mathbb{R}$,
by definition the Ruelle operator
$\mathcal{L}_A:C(M)\to C(M)$
acts on continuous functions $\phi:M\to\mathbb{R}$, in such way that, $\mathcal{L}_{A} (\phi)=\varphi$, where
$$
	\varphi(x)=\mathcal{L}_{A} (\phi)(x)
	=
	\sum_{T(y)=x} e^{ A(y)} \, \phi(y).
$$

This operator (sometimes called transfer operator) helps to understand equilibrium states
in Thermodynamic Formalism.  This corresponds to the analysis of the Statistical Mechanics of the one-dimensional lattice at positive temperature (see \cite{PP}). Maximizing probabilities correspond to the limit of equilibrium states when temperature goes to zero (ground states) as one can see for instance in \cite{BLLco}.

When $A$ is such that $\mathcal{L}_{A}(1)=1$ we say that $A$ is normalized.

The dual operator $\mathcal{L}_{A}^*$ acts on the space
of probabilities measures on $M$. Given a probability  $\mu$, then, $\mathcal{L}_{A}^*(\mu)=\nu$
where
the probability measure $\nu$ is the unique
one satisfying
$$
	\int \, \phi \,\,d\, \mathcal{L}_{A}^*\, (\mu)
	=
	\int \phi\, d\nu =
	\int  \mathcal{L}_{A}\,(\phi)\, d \mu
$$

for any continuous function $\phi$.

An important result claims that there exists a positive value $\lambda$  which is simultaneous an eigenvalue for  $\mathcal{L}_{A}$ and $\mathcal{L}_{A}^*$
(see \cite{PP}). This $\lambda$ is the spectral radius of $\mathcal{L}_{A}$. This defines a main eigenfunction for $\mathcal{L}_{A}$ and a main eigenprobability for $\mathcal{L}_{A}^*$.

In \cite{KLS} it is shown that the dual of the Ruelle operator $\mathcal{L}_{A}^*$ is a contraction for the $1$-Wasserstein distance when $A$ is normalized. The fixed point probability is the main eigenprobability for $\mathcal{L}_{A}^*$.
\medskip

We suppose that $c$ is a normalization constant for $W$ in the
sense that

\begin{equation} \int\, \int \, e^{W(x,y)-c} \, d \nu_{A^*} (y) \,d \nu_{A} (x)\,=\,1,
\end{equation}

where $\nu_A$ and $\nu_{A^*}$ are respectively the eigen-probability
for the dual Ruelle operator of  $A$ and $A^*$ \cite{CLT}.

We also denote by $\phi_A $ and $\phi_{A^*}$ the corresponding
eigen-functions for $\mathcal{L}_{A}$. Finally,
 $\mu_A   = \nu_{A}  \, \phi_{A} =  $ and $  \mu_{A^*} = \nu_{A^*} \, \phi_{A^*}  $
are the invariant probabilities which are the solutions of the
respective pressure problems for $A$ and $A^*$.

For a fixed $A$ we consider a real parameter $\beta$, and the corresponding potentials $\beta A$, and the
eigenfunctions $  \phi_{\beta\,A} $, and so on.

In Statistical Mechanics $\beta$ is the inverse of temperature. In this way asymptotic results when $\beta\to \infty$ can be consider
as the ones which describes the system in equilibrium  at temperature zero.

Note that $\beta W$ is an involution kernel   for $\beta A$, and its dual is $\beta A^*$.

It is known (see for instance \cite{CLT}) that a sub-action  $V$ can
obtained as the limit
\begin{equation}
 V(x) = \lim_{\beta \to \infty} \frac{1}{\beta} \, \log
\phi_{\beta A} (x).
\end{equation}

This $V$ is a calibrated sub-action for $A$ (see \cite{CLT}
\cite{BLT} \cite{GL1}).

We can also get a calibrated
sub-action $V^*$ for $A^*$ using the limit
\begin{equation}
V^*(y) = \lim_{\beta \to \infty} \frac{1}{\beta} \, \log
\phi_{\beta A^*} (y)\,\, . \end{equation}

From  \cite{BLT} (see also \cite{LMMS1}) we have
$$ \phi_{ A^*} (y) = \int e^{\, W_A(x,y) - c}
\, d \nu_{ A} (x) .$$

Finally, we define for each $x\in \Sigma$,
$$I(x) =\sum_{n=0}^\infty \, [ \,V\, \circ \,\sigma - V - (\,A-m(A))\,]\,
\sigma^n\, (x),
$$
where $V$ is a (any) calibrated sub-action.

The function $I$, where $I:\Sigma \to \mathbb{R} \cup \{\infty\}$, can have infinite values, but it is lower semi-continuous.

In \cite{BLT} it is shown that for any cylinder set $C\subset \Sigma$,
$$
\lim_{\beta \to +\infty} \frac{1}{\beta} \log
\mu_{\beta\, A}(C)=-\inf_{ x \in C} I(x)
$$

In this way we get a Large Deviation principle for $\mu_{\beta \, A} \to \mu_\infty.$

Remember that we denote
by $\mu_{\infty}^*$ the unique maximizing probability for $A^*$ (it is unique because $\mu_\infty$ is unique for $A$, and, moreover, $A$ and $A^*$ are cohomologous in $\hat{\Sigma}$).

All the results described above are true for expanding
transformations $T$ of class $C^{1+\alpha}$  on the
circle $S^1$. In this case we have to consider the natural extension $\hat{T}$ of $T$. This also includes the case of $T(x)=-\, 2 x$ (mod 1).

In the case
$T: S^1 \to S^1$, given by $T(x)= \,2\,x $ (mod 1), we define ${\hat T}$
in the following way: the Baker transformation associated to $T$, denoted by $\hat{T}(x_1,x_2)$, where $\hat{T}:[0,1]^2\to [0,1]^2 $, is such that satisfies for all $(x_1,x_2)\in [0,1]^2$,
$\hat{T}(x_1, T^*(x_2))=(T(x_1), x_2)$ (see picture bellow) . In this case $T^* : S^1 \to S^1$, with $T^*(y) =  \,2\,y $ (mod 1), $\hat{T}$ plays the role of $\hat{\sigma}$, and $T^*$ plays the role of $\sigma^*$,  on the definitions and results above.

All the above apply for an expanding transformation $T: S^1 \to S^1$, or  $T:[0,1]\to[0,1]$

The transformation $\hat{T}$ on $S^1 \times S^1$, contract vertical fibers by forward iteration and expand (and cut) vertical fibers by backward iteration.

\begin{center}
\includegraphics[scale=0.6,angle=0]{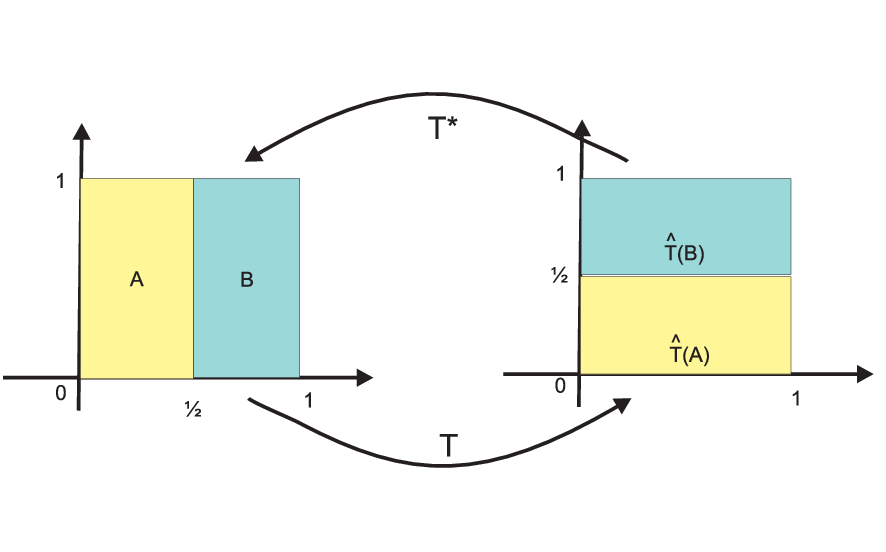}\\
\small{Characterization of $S$}  \\
\end{center}

Remember that we said that  $W: \hat{\Sigma}=\Sigma\times \Sigma \to \mathbb{R}$ satisfies the twist condition on
$ \hat{\Sigma}$, if  for any $(a,b)\in \hat{\Sigma}=\Sigma\times
\Sigma $ and $(a',b')\in\Sigma\times \Sigma $, with $a'> a$,
$b'>b$, we have
\begin{equation}
W(a,b) + W(a',b')  <  W(a,b') + W(a',b).
\end{equation}

We have the analogous definition for expanding transformations on the interval:

\begin{definition} We say $W: [0,1]^2 \to \mathbb{R}$ continuous satisfies the twist condition on
$  [0,1]^2 $, if  for any $(a,b)\in [0,1]^2 $ and $(a',b')\in [0,1]^2  $, with $a'> a$,
$b'>b$, we have
\begin{equation}
W(a,b) + W(a',b')  <  W(a,b') + W(a',b).
\end{equation}

\end{definition}

Same definition for $W$ on $S^1 \times S^1$.

When $x,y\in [0,1]$ (or, on $S^1$), the condition
$$ \frac{\partial^2\, W(x,y)}{\partial x \,\partial y}<0,$$
implies the twist condition for $W$.

The twist condition can be seen as a kind of transversality condition (see \cite{LO})
\vspace{0.3cm}
\begin{exemp}

Consider the transformation $T: S^1 \to S^1$, given by $T(x)= -\,2\,x $ (mod 1) and  $A(x)= a + bx + c x^2$, where
$a,b,c$ are constants and $c>0$. In item b) in the appendix we show an explicit expression for the $W$-kernel and we prove that $W$  satisfies the twist condition. From this, we can get an explicit expression for the calibrated subaction for a certain potential (see remark 6 in the appendix).

We point out that for considering the system above in $S^1$ we have to assume above that $A(0)=A(1).$ If we are interested in the case of $[0,1]$ the same result can be obtained but we do not have to assume  $A(0)=A(1).$

Moreover, we also show in item c) in the appendix  that a certain class of analytic perturbations of $A(x)= a + bx + c x^2$ produces $W$-kernels which are twist.

\end{exemp}

\begin{exemp} In  item d) in the appendix  we show an example of a $W$-kernel for a continuous potential $A$,  and for the action of the shift $\sigma$ on the Bernoulli  space $\{0,1\}^\mathbb{N}$, which is twist.
\end{exemp}

\begin{exemp}

 Consider the Gauss map $T(x)= \frac{1}{x} - [\frac{1}{x}]$ on $[0,1]$.

  We can define the Baker transformation associated to $T$, denoted by $\hat{T}(x_1,x_2)$, where $\hat{T}:[0,1]^2\to [0,1]^2 $.

 The $W$ kernel for $A(x_1)=-\log T'(x_1)$, which is $W(x_1,x_2) = - 2\, \log (1 \,+\, x_1\,x_2)$
 (see \cite{BLT}).

 It is known that the dual of $A=-\log T'$ is $A^*=-\log T'$ (see Proposition 4 in \cite{BLT}).

The maximizing probability for such potential $-\log T'(x)= 2 \log (x)$ is the $\delta$-Dirac in the fixed point $b$, where $b$ is  the golden mean $b=\frac{\sqrt{5}-1}{2}$ (see for instance \cite{CG}). In this case $m(A)= 2 \log (b)$.

Note that $W$ is differentiable on any point $(x_1,x_2) \in  [0,1]^2$.

One can easily see that an explicit calibrated sub-action $u$ (unique up to an additive constant because the maximizing probability is unique \cite{GL1})
satisfying
\begin{equation}
u(x) \le u(T( x) ) - A(x) + m(A),
\end{equation}
is $u(x)= W(x,b) = - 2\, \log (1 + x \,b )$.

Note that
$$ \frac{\partial^2\, W(x,y)}{\partial x \,\partial y}<0,$$
and, therefore, $W$ is twist.

\end{exemp}

\begin{exemp}
Suppose $T(x)$ is $-\,2\,x$ (mod 1), $T: [0,1]\to [0,1]$ and $A:[0,1]\to \mathbb{R}$ is Holder and monotonous. Under some assumptions on $A$ one can get cases where the maximizing probability is unique and with support  on the right fixed point $p$ (see \cite{JS}). In the same way as in last example one can show that $V(x)= W(x,p)$ is a calibrated subaction.

If one considers on the interval $[0,1]$ the potential $A(x)= x^2$ then  we are under such assumptions. One can show that $A^* (y)=y^2$, and $W(x,y)=  (1/3)(x^2+y^2) - (4/3)xy\,$ (see remark 5 in item b) in the appendix). In the same way $ \frac{\partial^2\, W(x,y)}{\partial x \,\partial y}<0.$

\end{exemp}

\begin{exemp}

Consider the transformation $T: S^1 \to S^1$, given by $T(x)= -\,2\,x $ (mod 1) and  $A(x)=-(x-\frac{1}{2})^{2}$ (a continuous potential on $S^1$) for which all results in \cite{BLT} apply (see also \cite{LO} where it is shown in this case  the graph property).

The maximizing probability has support in the periodic orbit of period $2$ (see \cite{J3} and \cite{J6}).

 One can define the continuous Baker transformation associated to $T$, denoted by $\hat{T}(x_1,x_2)$, where $\hat{T}:[0,1]^2\to [0,1]^2 $ is such that satisfies for all $(x_1,x_2)\in [0,1]^2$,
 $\hat{T}(x_1, T(x_2))=(T(x_1), x_2)$.


 In this case, we show in remark 6 in the appendix that  a smooth $W$-kernel is:
$$W(x,y)=-(1/3) x^{2}-(1/3)y^{2}+(4/3)xy-(2/3)x-(1/3)y.$$

The dual potential $A^*$ is equal to $A$.

This $W$-kernel is {\bf not} twist  because  $ \frac{\partial^2\, W(x,y)}{\partial x \,\partial y}>0.$

It follows from a general result presented in \cite{JS} that any maximizing measure for this potential is $\mu_{\infty}=(1-t)\delta_{1/3}+ t \delta_{2/3}$, where $t \in [0,1]$, so the critical value is $m=A(1/3)=A(2/3)$.

It is easy to verify that,
\begin{align*}
V(x) &= (W(x,1/3) - W(1/3,1/3)) \chi_{[0,1/2)}(x) + \\
& W(x,2/3) - W(2/3,2/3) \chi_{[1/2,1]}(x)\\
&=\max\{ W(x,1/3) - W(1/3,1/3),  W(x,2/3) - W(2/3,2/3)\}
\end{align*}
is a calibrated subaction for $A$.
\begin{center}
\includegraphics[scale=0.6,angle=0]{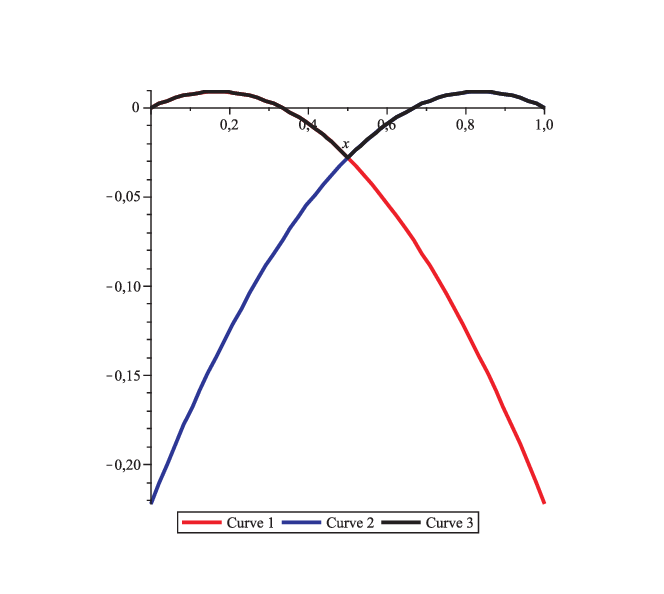}\\
\small{$W(x,1/3) - W(1/3,1/3)$=red, $W(x,2/3) - W(2/3,2/3)$=blue and $\phi$=black - The calibrated subaction is the supremum of the two functions described in the picture.}\\
\end{center}

This calibrated subaction is not analytic but piecewise analytic (see \cite{LOS} for more general results).

\end{exemp}

\begin{exemp}

Consider the transformation $T: S^1 \to S^1$, given by $T(x)= -\,2\,x $ (mod 1) and  $A(x)=(x-\frac{1}{2})^{2}$ (a continuous potential on $S^1$) for which all results in \cite{BLT} apply.

 In this case we show  in item b) in the appendix that a smooth $W$-kernel is:
$$W(x,y)=(1/3) x^{2}+(1/3)y^{2}-(4/3)xy+(2/3)x+(1/3)y.$$

The dual potential $A^*$ is equal to $A$.

This involution kernel $W$ is  twist.

\end{exemp}

Similar results can be obtained for
$T: S^1 \to S^1$, given by $T(x)= \,2\,x $ (mod 1) and  $A(x)=-(x-\frac{1}{2})^{2}$ (a continuous potential on $S^1$)

\begin{definition} Given $G:\hat{\Sigma} \to \mathbb{R}$ upper semi-continuous, and
$f(x)$ continuous, where $f:\Sigma \to \mathbb{R}$,  we
define the $G$-transform of $f$, denoted by $f^\# (y)$, where $f^\#
: \Sigma^* \to \mathbb{R},$ the function such that

\begin{equation}
f^\# (y) =  \max_{x \in \Sigma} \, \{- f(x) + G(x,y) \} .
\end{equation}
We can use also the notation $f^{\#}_G$, instead of $f^\# ,$ if we want to stress the dependence on $G$.
\end{definition}

In this case we say that $f^\#$  is the $G$-conjugate of $f$ \cite{Vi1} \cite{Vi2}.
We use the notation of \cite{R} page 268.

Note that, if we add a constant to $f$, then new $f^\#$ will be
obtained from the old one by subtracting the same constant.
Therefore, in this case the sum $f(x) + f^\# (y)$ will be the same.

We are interested, for example, when $G=-W$ or $G=-W+I$.

A similar definition and properties can be consider for expanding transformations on $[0,1]$.

\begin{proposition} If $V$ is a subaction for $A$, then $V^\#  =V^\#_W$ is a
subaction for $A^*$.
\end{proposition}

{\bf Proof:} Given $y$ there exist $z^0$ such that

$$ V^\# ( \sigma^* (y) ) - V^\# (y)= \max_{x \in \Sigma} \, \{- V(x) + W(x, \sigma^* (y))  \}
- $$

$$  \max_{z \in \Sigma} \, \{- V(z) + W(z,y)  \} =$$

$$ \max_{x \in \Sigma} \, \{- V(x) + W(x,\sigma^* (y) ) \}
-  \, (\,- V(z_0) + W(z_0,y)\,) \geq $$

$$ - V(\tau_y (z_0) ) \,+\, W(\tau_y (z_0),\sigma^* (y) ) ) \,+\, V(z_0) \,- \,W(z_0,y)  \geq $$

$$ A ( \tau_y (z_0)  ) - m(A) +  W(\tau_y (z_0), \sigma^* (y) ) \,-
\,W(z_0,y)\,  =$$

$$ A^* (y) - m (A)  \,           =
A^* (y) - m (A^* )    .$$

\qed

The subaction you get by $-W$-transform is not necessarily calibrated.

Note that if we add a constant to $W$ (the new $W$ will be also a
$W$-Kernel), then all of the above will be also true.

In a similar  way like in the reasoning of last proposition one can get:

\begin{proposition} If $V^*$ is a sub-action for $A^*$, then
$$(V^*)^\#_W\,(x)=   \max_{z \in \Sigma^*} \, \{- V^*(z) + W(x,z)  \} $$
is a subaction for $A$.
\end{proposition}

\bigskip

Analogous definitions can be consider for an expanding transformation  $T:S^1 \to S^1$. This also includes the case of $T(x)=-\, 2 x$ (mod 1).

\bigskip
\section{The transport problem}\label{secaosegunda}

We assume the maximizing probability $\mu_\infty$ for $A$ is unique.

We denote by $\mu_\infty^*$ a fixed maximizing probability for
$A^*$.

We denote by ${\cal K}( \mu_\infty, \mu_{\infty}^* )$ the set of
probabilities $\hat{\eta} (x,y)$ on $\hat{\Sigma}$, such that

$$ \pi_ x^*  (\hat{\eta} ) = \mu_{\infty}, \, \, \text{and}\,\,\pi_ y^*  (\hat{\eta} ) =
\mu_{\infty}^* \,.$$

We are going to consider bellow the cost function $c(x,y) =I(x) - W(x,y) + \gamma  ,$ which is defined for $x$ such that $I(x)\neq \infty$.

{\bf The Kantorovich Transport Problem:} Given $A$ (and all the
probabilities described above) we are interested in the minimization
problem

$$ C ( \mu_\infty, \mu_{\infty}^* )   \,=\,  \inf_{\hat{\eta} \in {\cal K}( \mu_\infty, \mu_{\infty}^* ) }  \int \int
(I(x) - W(x,y) + \gamma ) \, d\,\hat{\eta} \,= $$ $$\,
\inf_{\hat{\eta} \in {\cal K}( \mu_\infty, \mu_{\infty}^* ) }  \int
\int c(x,y) \, d\,\hat{\eta} =   $$

\begin{equation}
\,  \max_{\hat{\eta} \in {\cal K}( \mu_\infty, \mu_{\infty}^* ) }  \int \int
(W(x,y) - \gamma - I(x) ) \, d\,\hat{\eta} \,
\end{equation}

where, $I$ is the deviation function for $\mu_\infty= \lim_{\beta
\to \infty} \, \mu_{\beta A}$ (see \cite{BLT}),

\begin{equation}
c_\beta=\int \int e^{ \beta\,
W(y,x)} \, d \nu_{\beta A} (x) \, d \nu_{\beta A^* } (y),
\end{equation} and

\begin{equation}
\gamma =\lim_{\beta \to \infty} \frac{1}{\beta}    \log c_\beta           \, ,
\end{equation}
 as in
proposition 5 in \cite{BLT}.

We call $c(x,y) =  - W(x,y) + \gamma + I(x)$ the cost function. Therefore, $c$ is lower semi-continuous.

A probability $  \hat{\eta} $ on $\hat{\Sigma} $ which attains such
minimum is called an optimal transport probability. We denote it by $\hat{\mu}$.

We will show later that $\hat{\mu}_{max}$, the natural extension of
$\mu_\infty$, will be the optimal transport probability $\hat{\mu}$.

\medskip

One of our main results is Theorem \ref{mai} which claims that:

The support of
$\hat{\mu}_{max} $ is $c$-cyclically monotone. In other words, the the twist condition for $c$ is true when restricted to the support of the maximizing probability $\hat{\mu}_{max}$.

\medskip
 {\bf
Remark 1:} Note that if we subtract the deviation function $I(x)$ of
the cost function, that is, if we consider a new cost $c(x,y)
=-W(x,y)+ \gamma $, the problem above will not change, because $I$
is constant zero in the support of $\mu_\infty$.

In other words
$$C ( \mu_\infty, \mu_{\infty}^* )   \,=\,  \inf_{\hat{\eta}
\in {\cal K}( \mu_\infty, \mu_{\infty}^* ) }  \int \int (- W(x,y) +
\gamma ) \, d\,\hat{\eta} \,, $$

and, the optimal transport probability will be the same.

In some sense this setting is nicer because the cost $c$ is a continuous function on $\hat{\Sigma}$.

\begin{definition} A pair of functions $f(x)$ and $f^\#(y)$ will be
called $c$-admissible (or, just admissible for short) if
\begin{equation}
f^\# (y) =  \min_{x \in \Sigma} \, \{- f(x) + c(x,y) \} \, \,
\end{equation}
\end{definition}

In other words $-f^\#$ is the $-c$-conjugate of $-f$.

Note that in this case, $\forall x \in \Sigma,\, y\in \Sigma^*$, we have that
$$f(x) + f^\# (y) \leq c(x,y) .$$

We denote by ${\cal F}$ the set of all admissible pairs   $(f(y)
,f^\# (y))$.

\medskip

{\bf The Kantorovich dual Problem:} Given $A$ and the corresponding  $c$ ($W$ and all the
probabilities described above) we are interested in the maximization
problem
\begin{equation}
D (\mu_\infty  , \mu_{\infty}^*  )   \,=\,  \max_{(f, f^\#) \in {\cal F} }\, (\, \int f d \mu_{\infty} +  \int
 f^\# d \mu_{\infty}^* \,)   .
 \end{equation}

\bigskip

A pair of admissible $(f, f^\#)\in {\cal F}$ which attains the
maximum value will be called an optimal pair.

The Kantorovich duality theorem (see \cite{Vi1}) claims that under general conditions  $D (\mu_\infty  , \mu_{\infty}^*  ) = C ( \mu_\infty, \mu_{\infty}^* ).$

The main tool to prove this result is the Fenchel-Rockafellar duality Theorem.

\begin{theorem}[\textbf{Fenchel-Rockafellar duality}]\label{FR}
\label{Fenchel}
Suppose $E$ is  a normed vector space,  $\Theta$ and $\Xi$ two convex functions defined on $E$ taking values in $\mathbb{R}\cup \{+\infty\}$. Denote $\Theta^{\ast}$ and  $\Xi^{\ast}$, respectively, the Legendre-Fenchel transform of  $\Theta$ and $\Xi$.
Suppose there exists  $v_0\in E$, such that $\Theta(v_0)<+\infty,\, \Xi(v_0)<+\infty$ and that $\Theta$ is continuous on $v_0$.

Then,
\begin{equation}
\inf_{v \in E}[\Theta(v)+\Xi(v)]=\sup_{f\in E^{*}}[-\Theta^{*}(-f)-\Xi^{*}(f)] \label{rockafeller}
\end{equation}
Moreover, the supremum in ($\ref{rockafeller}$) is attained in at least one element in $E^*$.
\end{theorem}

We will not present the proof of this general theorem but we will present a nice geometric proof in a simple case (one-dimensional) in item e) in the Appendix.

\bigskip

We suppose, from now on, that the maximizing probability for $A$,
denoted by $\mu_\infty$ is unique.

We denote, as in \cite{CLT} the calibrated sub-actions  $V$ and
$V^*$ by
\begin{equation}
 V(x) = \lim_{\beta \to \infty} \frac{1}{\beta} \, \log
\phi_{\beta A} (x)\,\, \mbox{and} \,\, V^*(y) = \lim_{\beta \to
\infty} \frac{1}{\beta} \, \log \phi_{\beta A^*} (y)\,\, .
\end{equation}

The above convergence is uniform and $V$ is (up to constant) the unique calibrated sub-action for $A$ (see \cite{CLT} \cite{BLT} \cite{GL1}).

We will show later that $(f,f^\#)$ such that $f(x)= -V(x)$ and $f^\#
(y) = - V^* (y)$ is the optimal pair.

\bigskip

{\bf Important property:} If $\hat{\mu}$ is an optimal transport
probability and if $(f,f^\#)$ is an optimal pair in ${\cal F}$, then
the support of $\hat{\mu}$ is contained in the set

\begin{equation}
\{\, <y,x>\,\, \in \hat{ \Sigma} \,| \,\mbox{such that}\,\, (f(x) + f^\#
(y))\, = \, c(x,y)\,\}.
\end{equation}

 It follows from the prime and dual linear programming
problem formulation. The condition above is the complementary
slackness condition (see \cite{EG} \cite{Ra} \cite{GM}).

\bigskip

The reciprocal of this result is also true (see \cite{Vi2} Remark 5.13 page 59).

If $x$ and $y$ are such that  $(f(x) + f^\#
(y))\, = \, c(x,y)$ we say that they are realizers for the cost $c$. In \cite{CLO} it is shown that the set of realizers for $I-W$  is an invariant set for the dynamics of
$\hat{\sigma}.$ In this section we are mainly concerned with the support and not with all realizers.

If one finds $\hat{\mu}$ an an admissible pair  $(f,f^\#)$
satisfying the above claim (for the support), then, one solves the
Kantorovich problem, that is, one finds the optimal transport
probability $\hat{\mu}$ .

No we will prove Theorem 1.

\begin{proposition}
The minimizing Kantorovich
probability $\hat{\mu}$ on $\hat{\Sigma}$  associated to $-W$  is
$\hat{\mu  }_{max}$.

\end{proposition}

{\bf Proof:}
Proposition 10 (1) in \cite{BLT} claims that if $\hat{\mu}_{max}$ is the
natural extension of the maximizing probability $\mu_\infty$, then
for all $<p^* | p>$ in the support of $\hat{\mu}_{max}$ we have

$$- V(p)\, - \,V^* (p^*)\,= \, - W(p,p^* )  \,+\, \gamma .$$

This is the same as saying that in the support of $\hat{\mu}_{max}$

$$- V(p)\, - \, V^* (p^*)\,=\,  - W(p,p^*)  \,+ \,\gamma\, +\, I(p)\,=\,c(p,p^*) ,$$

because $I$ is zero in the support of $\mu_\infty.$

Then if $-V(x)$ and $-V^* (y)$ is an admissible pair, then
$\hat{\mu}_{max}$ is the optimal transport probability for such
$c(x,y)$. This will be shown in the next proposition.

We will show bellow that the $-c$-transform of $V$ is $V^*$.

\qed

Note that if $W$ is a $W$-Kernel for $A$, for all $\beta$, we have
that $\beta W$ is a $W$-Kernel for $\beta A$. We denote by $c_\beta$
the normalizing constant for $\beta W$, as in \cite{BLT}. It is known that
$ \frac{1}{\beta} \log \, c_\beta = \gamma$.

Now we will show Theorem 2.

\begin{proposition} The pair
$(-V,-V^*)$ is admissible.
\end{proposition}
{\bf Proof:} For a fixed $y$ we have to show that
$$-V^* (y)\, =\, (-V)^\#_c\,=\,
\inf_{x \in \Sigma} \, \{ -(-V(x)) + c(x,y) \} \,.$$

This is the same as

$$V^* (y) = \sup_{x \in \Sigma} \, \{ \,(-V(x)) - c(x,y) \, \}= \sup_{x
\in \Sigma} \, \{ \, -V(x) - (\gamma - W(x,y) + I(x)   \,)   \}
\,,$$

or, for all $x$
\begin{equation}
-V^* (y)\,\leq V(x) + c(x,y) \,.
\end{equation}

From proposition 3 in \cite{BLT} (we just write here $W(x,y)$,
instead of $W(y,x)$ there) we have
$$ \phi_{\beta A^*} (y) = \int e^{\beta\, W_A(x,y) - c_\beta}
\frac{1}{ \phi_{\beta A} (x)   } \, d \mu_{\beta A} (x) =$$ $$\int
e^{\beta\, W_A(x,y) - c_\beta -\log \phi_{\beta A} (x) }  \, d
\mu_{\beta A} (x) .$$

Consider now the limit

$$V^* (y)\,=\, \lim_{\beta \to \infty} \frac{1}{\beta} \, \log (\phi_{\beta A^*}
(y))
=$$
$$\lim_{\beta \to \infty} \frac{1}{\beta} \, \log \int
e^{\beta\, W_A(x,y) - c_\beta -\log \phi_{\beta A} (x) }  \, d
\mu_{\beta A} (x) .$$

From \cite{CLT} the function $\frac{1}{\beta} \, \log (\phi_{\beta
A} (x))$ converges uniformly with $\beta$ to $V(x)$.

Therefore, one can write
$$\lim_{\beta \to \infty} \frac{1}{\beta} \, \log \int
e^{\beta\, W_A(x,y) - c_\beta -\log \phi_{\beta A} (x) }  \, d
\mu_{\beta A} (x) = $$ $$  \lim_{\beta \to \infty} \frac{1}{\beta}
\, \log \int e^{\beta\, (\, W_A(x,y) - \gamma - V (x) \,)} \, d
\mu_{\beta A} (x)          $$

Now, by Varadhan's Integral Lemma \cite{DZ} we obtain
$$V^* (y)\,= \sup_x \{W_A(x,y) - \gamma - V (x) - I(x)\}=\sup_x \{- V(x) + W(x,y) -\gamma  - I(x)\}  ,$$

where $I$ is the deviation function.

\qed

Finally, we get that
$\hat{\mu}_{max}$ is the optimal transport probability for such
$c(x,y)$. From now on we will use either the notation $\hat{\mu}$ or $\hat{\mu}_{max}$ for the optimal transport probability.

In \cite{LOS}  Transport Theory is used as a tool to show that in some cases the calibrated subaction is piecewise analytic.
In \cite{CLO} some generic properties of the potential $A$ is considered and special results about  the realizers of the
$W-I$ are obtained.

The last theorem says: for any $y\in \Sigma^*$ we have

\begin{equation}V^* (y)\,= \sup_{x\in \Sigma} \{- V(x) - c(x,y) \}   .
\end{equation}

Note that when $y=p^*$, for $p^*$  in the support of $\mu_{\infty}^*$, the  supremum
$$V^* (p^*)\,=\sup_x \{- V(x) + W(x,p^*) -\gamma  - I(x)\}= \sup_x \{- V(x) - c(x,p^*) \}   ,$$
is realized at $x=p$, for $p$ in the support of $\mu_\infty$ (with
$<p^*,p>$ in the support of $\hat{\mu}$).
\vspace{0.2cm}

{\bf Remark 2:}
Remember that, if the maximizing probability for $A^*$ is unique,
then there is a unique calibrated sub-action for $A^*$ (up to
additive constant) \cite{BLT} \cite{GL1}.

\vspace{0.2cm}

Analogous definitions and properties can be obtained for $T:S^1 \to S^1$. This also includes the case of $T(x)=-\, 2 x$ (mod 1).

We could likewise consider the analogous problem for  $A^*$: given $A^*$ (obtained from $A$) fixed, denote $I^{*}:\Sigma^* \to \mathbb{R}$, the non-negative deviation function for $\mu_{\beta \, A^* }\to \mu_\infty^*$.

Denote $c^* (x,y) =(I^*(y) - W(x,y) + \gamma ) $.

Then, consider the problem
$$ C ( \mu_\infty, \mu_{\infty}^* )   \,=\,  \inf_{\hat{\eta} \in {\cal K}( \mu_\infty, \mu_{\infty}^* ) }  \int \int
(I^*(y) - W(x,y) + \gamma ) \, d\,\hat{\eta} \,= $$ $$   \,
\inf_{\hat{\eta} \in {\cal K}( \mu_\infty, \mu_{\infty}^* ) } c^* (x,y)\, d\,\hat{\eta}   =     \,
\inf_{\hat{\eta} \in {\cal K}( \mu_\infty, \mu_{\infty}^* ) }  \int
\int (-W(x,y)+ \gamma) \, d\,\hat{\eta} ,   $$
which have the same minimizing measures, as for the minimization for $c(x,y) =(I(x) - W(x,y) + \gamma ) $ among probabilities on $ {\cal K}( \mu_\infty, \mu_{\infty}^* )$.

Note also that from proposition 3 in \cite{BLT}  we have
$$ \phi_{\beta A} (x) = \int e^{\beta\, W_A(x,y) - c_\beta}
\frac{1}{ \phi_{\beta A^*} (y)   } \, d \mu_{\beta A^*} (y) =$$ $$\int
e^{\beta^*\, W_A(x,y) - c_\beta -\log \phi_{\beta A^*} (y) }  \, d
\mu_{\beta A^*} (y) .$$

In the same way as before one can show that
for any $x\in \Sigma$, we have
\begin{equation}
V(x)=(-V^*)^\#_{c^*}\,= \sup_{y\in \Sigma^*} \{- V^*(y) - c^*(x,y) \}   .
\end{equation}

Note that $c(x,y)= c^* (x,y)$ in the support of the minimizing  $\,\hat{\mu}_{max}$ for $c$ (or for $c^*)$ .

{\bf Remark 3:}
It is not necessarily true that $   (\,(-V^*)^\#_{c^*}\,)^\#_{c^*}= -V^*.$ However, the expression is true when restricted to the support of the optimal transport probability $\hat{\mu}_{max}$.
In the same way  $   (\,(-V)^\#_{c}\,)^\#_{c}= -V$ in the support of $\hat{\mu}_{max}$.

\bigskip
\section{Graph properties and the twist condition}\label{secaog}

Consider a lower semi-continuous continuous cost function  $c(x,y)$ on $\hat{\Sigma}$ (or, a continuous cost function  $-W(x,y)$ on $\hat{\Sigma}$). We refer the reader to
\cite{Ra} \cite{Vi1} \cite{Vi2} and {\cite{GM} for general references on transport mass problems.

\begin{definition}
A set $S\subset \hat{\Sigma}$ \,is called $c$-cyclically
monotone, if for any finite number of points $(x_j,y_j)$ in $S$, $j
\in \{1,2,...,n\}$, and any permutation $\sigma$ of the $n$ letters, we have
\begin{equation}
 \sum_{j=1}^n c( x_j,y_j) \,\leq\, \sum_{j=1}^n c( x_{\sigma
(j)},y_j).
\end{equation}
\end{definition}

\begin{proposition}(see Theorem 2.3 \cite{GM}).\, For a continuous
function $c(x,y) \geq 0,$ where $\hat{\Sigma}$, if  $\rho \in {\cal K}
(\mu_{\infty}, \mu_{\infty}^*)$ is optimal for $c$, then, $\rho$ has
a $c$-cyclically monotone support.
\end{proposition}

 \begin{corollary} The support of $\hat{\mu}_{max}$, the
 natural extension of $\mu_{\infty}$ is $c$-cyclically monotone.
 \end{corollary}
\vspace{0.3cm}
We will present bellow in the next theorem a direct proof of this fact.
\vspace{0.3cm}

\begin{definition}
A function $f: \Sigma \to \mathbb{R} \cup \{\infty\}$
is $c$-concave, if there exist a set $A\subset \Sigma \times \mathbb{R}$
such that
$$ f(y) = \sup_{(x,\lambda) \in A} \{ c(x,y) + \lambda\}$$
\end{definition}

\begin{definition}
A function $f: X \to \mathbb{R} \cup \{\infty\}$
is $c$-convex, if $(-f)$ is $c$-concave.
\end{definition}



\begin{definition}
Given $x\in \Sigma$, the set $\hat{\partial}_c\, f(x)$ is the set of  $y\in\hat{\Sigma}$ such that, for
all $z\in \Sigma$ we have

$$f(z)-f(x) \leq c(z, y) - c(x,y )$$

In this case we say $y$ is a $c$-sub-derivative for $f$ in $x$.

\end{definition}

An important problem is to know, for a certain given $x$, if the $\hat{\partial}_c\, f(x)$ has cardinality $1$.



\begin{proposition} (see Theorem 2.7 in \cite{GM}, Lemma 2.1 in
\cite{R} and section 4 in \cite{Ra}). For $S\subset \hat{\Sigma}$ to be $c$-cyclically monotone, it is
necessary and sufficient that $S \subset \hat{\partial}_c (f)(x) =\{(x,y)\,|\, f(z) - f(x) \leq c(z,y) -
c(x,y)\, ,\, \forall z \in X \}$, for
some $c$ concave $f$, where $f : \Sigma \to \mathbb{R}\cup\{\infty\}$.

Moreover: $f$ is defined in the following way: choose
$(x_0,y_0)\in S$, then

$$ f(x)= \, \inf_{n \in \mathbb{N}, \, (x_j,y_j) \in S, \, 1\leq
j\leq n}\,\,[\, (\, c(x, y_n)- c(x_n,y_n)\,) +$$ $$ (\, c(x_n,
y_{n-1})- c(x_{n-1},y_{n-1})\,)+...$$ $$+(\, c(x_2, y_1)-
c(x_1,y_1)\,)+ (\, c(x_1, y_0)- c(x_0,y_0)\,) \,] .$$
\end{proposition}




We assume, without lost of generality that $m(A)=0$.
\vspace{0.2cm}

Note that if $S\subset \hat{\Sigma}$ is a graph, then for each $x\in \Sigma$ in the $x$-projection of $S$, we have that $\hat{\partial}_c (f)(x)$ has cardinality $1$.
\vspace{0.2cm}

Consider fixed $(x_0,y_0) ,(x_1,y_1)  $ in the  support of
$\hat{\mu}_{max}  $ and  $(x_0,y_1) ,(x_1,y_0) \in \hat{\Sigma}$.

Given a function $f(x,y)$ we denote
\begin{equation} \Delta_f \,((x_0,y_1) ,(x_1,y_0))\,=\,\,(\, f(x_0,y_0) +  f(x_1,y_1) )- (\, f(x_0,y_1)
+f(x_1,y_0)\,).
\end{equation}

Denote
\begin{equation}
b(x,y) = I(x) + \gamma -W(x,y) + V(x) + V^*(y).
\end{equation}

The $c$-cyclically monotone condition for the support of
$\hat{\mu}_{max} $ will follow from the claim

\begin{equation} \Delta_c \,((x_0,y_1) ,(x_1,y_0))\,=\,\,(\, c(x_0,y_0) +  c(x_1,y_1) )- (\, c(x_0,y_1)
+c(x_1,y_0)\,)\leq 0.
\end{equation}

This is so because any permutation of letters can be obtained by a
series of composition of transformations that exchange just two
letters.

It will follow from the proof bellow that $\Delta_c
 \circ\sigma=\Delta_c$

The next result does not assume a global assumption on twist condition for $c$.

\medskip

\begin{theorem} \label{mai}
\,Given $A:\Sigma \to \mathbb{R}$ Holder, then  $        c(x,y)\, = I(x) - W(x,y) + \gamma      \geq 0$, for all $(x,y)\in \Sigma$.
Moreover, for $(x_0,y_0) ,(x_1,y_1) $ in the support of
$\hat{\mu}_{max} $,  we have $ \Delta_c\leq 0.$
Therefore, the support of
$\hat{\mu}_{max} $ is $c$-cyclically monotone. In other words, the the twist condition for $c$ (or, for $W$) is true when restricted to the support of the maximizing probability $\hat{\mu}_{max}$.

\end{theorem}

{\bf Proof:} First we point out that $\Delta_c=\Delta_b$. We will
show that under our hypothesis is true that  $\Delta_b\leq 0$

First note that
$$ [\, V^* \circ \hat{\sigma}^{-1} - V^* - A^*\,]  \, \hat{\sigma}\, (x,y)=
[V^* - V^* \circ \hat{\sigma}-A - W + W \circ \hat{\sigma}] \,
(x,y)=$$
$$ [ \gamma+ V (x) + V^* (y) - W(x,y)] \, + \, [V\circ
\hat{\sigma}\,-V \, -\, A ] (x,y)\,  -$$
$$ [\gamma + V \circ \hat{\sigma} + V^*  \circ \hat{\sigma} - W \circ
\hat{\sigma}]\, (x,y).$$

 Remember (see \cite{BLT}) that
$$I(x) =\sum_{n=0}^\infty \, [ V\, \circ \,\sigma - V - A]\,
\hat{\sigma}^n\, (x,y)
$$

We denote
$$I_n (x,y) =\sum_{k=0}^{n-1} \, [ V\, \circ \,\sigma - V - A]\, \circ \,
\hat{\sigma}^k\, (x,y) = I_n (x),
$$
and

$$R_n\,(x,y) = \,I_n (x,y)\, + \, [ \,\gamma + V(x) + V^* (y)  - W(x,y)\,]\,-$$
$$ [\, \gamma +  V + V^*   - W\,]\,\hat{\sigma}^n\, (x,y).  $$

We claim that if $(x,y)$ is in the support of $\hat{\mu}_{max} $,
then $b(x,y)=0.$ Moreover, for all $(x,y) \in \Sigma$, we have
$b(x,y)\geq 0.$

One can prove this result by means of Varadhan's Integral Lemma
(\cite{DZ}) with  the same reasoning as in the last proposition of the previous section. We will give bellow a direct proof of the
claim.

Either $I(x)=\infty$, and the claim is trivially  true or $I(x)$ is
finite. In this case, any accumulation point of $\hat{\sigma}^n
(x,y)$ will be in the support of $\hat{\mu}_{max} $.

Moreover, $ b(x,y)=R(x,y) =\lim_{n\to \infty} \, R_n (x,y)\geq 0.$

As in the support of  $\hat{\mu}_{max} $, we have that $R(x,y)=0$,
then, $b(x,y)=0$.

In any case $ R(x,y)\geq 0.$ This shows the claim.

We point out that $\Delta_c = \Delta_b =\Delta_W$ in the case $I(x)$
is finite.

We also remark that if $ (x_0,y_0) $ is in support of
$\hat{\mu}_{max} $, then as $R(x_0,y_0)$ is zero, it follows that  $R(x_0,y)$ is
finite . This is so because $(x_0,y)$ is in the stable manifold of
$(x_0,y_0)$ and
$$ R_n ( x_0 ,y) - R_ n (x_0,y_0)\, =\,$$
$$
\sum_{k=1}^n \,\{\, [ V^* \circ \hat{\sigma}^{-1} - V^* - A^* ]
\hat{\sigma}^k (x_0,y)-   [V^* \circ \hat{\sigma}^{-1} - V^* - A^* ]
\hat{\sigma}^k (x_0,y_0)\,\} $$

Finally, if $(x_0,y_0)$ and $(x_1,y_1)$ are both in the support of
$\hat{\mu}_{max} $, then $R(x_0,y_1)< \infty$, $R(x_1,y_0)< \infty $
and $I(x_0)=0=I(x_1)$.

In this case, for any $(x,y)$ of the form
$(x_0,y_0),(x_1,y_1),(x_1,y_0)$, or $(x_0,y_1)$
$$ R(x,y)= I(x,y) + [\gamma  + V + V^* - W] (x,y)= b(x,y).$$

As we know that $R$ is non-negative, then
$$ [b(x_0,y_0)+ b(x_1,y_1)]\,- \,[b(x_1,y_0)+b(x_0,y_1)]\,=\,0\,-\,
[b(x_1,y_0)+b(x_0,y_1)]\leq 0.$$

This shows that $\Delta_b\leq 0.$

\qed

We  did not use the twist condition above.
\vspace{0.2cm}

Note that we could alternatively consider the function $g:\Sigma \to \mathbb{R}$ defined in the following way: choose
$(x_0,y_0)\in S$, then

$$ g(x)= \, \inf_{n \in \mathbb{N}, \, (x_j,y_j) \in S, \, 1\leq
j\leq n}\,\,[\, (\, W(x, y_n)- W(x_n,y_n)\,) +$$ $$ (\, W(x_n,
y_{n-1})- W(x_{n-1},y_{n-1})\,)+...$$ $$+(\, W(x_2, y_1)-
W(x_1,y_1)\,)+ (\, W(x_1, y_0)- W(x_0,y_0)\,) \,] ,$$
which has the advantage of just taking into account a continuous function $W$.

The graph property for $S=$ support of $\hat{\mu}$,  and all kinds of different considerations can be obtained from such $g$.

We want to show now that if $W$ satisfies the twist condition and the maximizing probability for $A$ is unique, then the support of $\hat{\mu}$ on $\hat{\Sigma}$ is a graph. Our proof works for the Bernoully space $\{0,1,2,..,d\}^\mathbb{N}$ as  well for the interval $[0,1]$ (considering $T$ either conjugated to $2 x $ (mod 1) or to $-2x $ (mod 1)).

Consider the cost $c(x,y)= I(x) - W(x,y) - \gamma$, and a subset $S \subset X \times Y$ $c$-cyclically monotone.

\begin{lemma} \label{characterize} Suppose the $c$ satisfies the twist condition and
let $S$ be a $c$-cyclically monotone subset, if $(a,b), (a',b') \in S$ and $a \neq a'$ and $b \neq b'$, then  $a < a'$ and $b > b'$, or  $a > a'$ and $b < b'$.
\end{lemma}
\begin{proof}
Indeed, suppose $a < a'$ then,  if $b < b'$,  the twist condition on $W$ implies that
$$c(a,b)+ c(a',b') > c(a,b')+ c(a',b).$$
On the other hand, $S$ is $c$-cyclically monotone subset, so
$$c(a,b)+ c(a',b') \leq c(a,b')+ c(a',b),$$
that is an absurd.
\end{proof}

A similar property is true for $W$.

This Lemma means that the correct figure associated to a pair of points in $S$ is given by:
\begin{center}
\includegraphics[scale=0.6,angle=0]{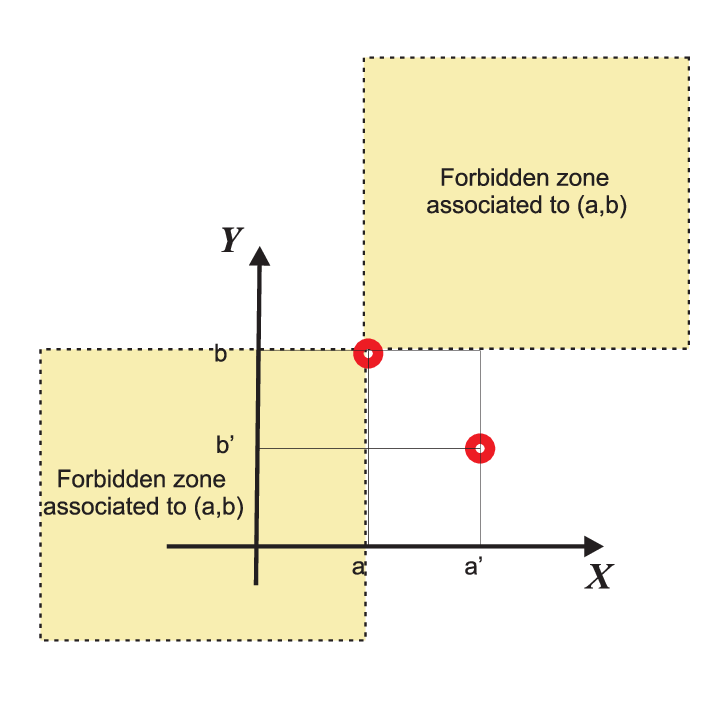}\\
\small{Characterization of $S$}  \\

\end{center}

We point out that, in principle, could exist points $z$ of $S$ in the vertical fiber passing by $a$ or in the
horizontal fiber passing by $b$.

Now we will show Theorem 3.

\begin{theorem}\label{GraphTheorem} (Graph Theorem) Suppose the $W$-kernel satisfies the twist condition and
let $\hat{\mu}$ be the $c$-minimizing measure of probability to the transport problem, then $S=\supp\,\, \hat{\mu}$ is a graph in $x$ (up to an orbit of measure zero), moreover this graph is monotone not increasing.
\end{theorem}
\begin{proof}

In order to get advantage of the geometrical and combinatorial arguments we will present pictures for the case of a transformation $T:[0,1]\to[0,1]$, given by $T(x)= 2\, x$ (mod 1).

Define $v^{+}(x)=\max\{y | (x,y) \in S\}$ and $v^{-}(x)=\min\{y | (x,y) \in S\}$.  In order to prove that $\supp\,\, \hat{\mu}$ is a graph we need to prove that $v^{-}(x)=v^{+}(x)$ for any $x$ in the support of $\mu_\infty$.

We say a point $(x,y)$ in the support of $\hat{\mu} $ is non-graph, if there exist another point of the form $(x,z)$,
in the support of  $\hat{\mu} $, and such that $z\neq y$.

Note that the image of two points in the support of $\hat{\mu}$ on the fiber over $x$ will go on two different points  in the support of $\hat{\mu}$ on the fiber over $\sigma(x)$. That is, the forward image by $\hat{\sigma}^n$ of non-graph points will go on  non-graph points. This maybe can not be true for backward images by $\hat{\sigma}^n$.

Suppose
the support of the maximizing probability $\mu_\infty$
(unique) is a periodic orbit. If $S$ is not a graph, then $v^{-}(x)<v^{+}(x)$ for some $x$.
As the transformation $\hat{\sigma}$ contracts each  fiber  by forward iteration, we have that, the image
of the interval fiber from $(x,v^{-}(x))$ to $(x,v^{+}(x))$, by a finite iterate of  $\hat{\sigma}$, goes  inside the fiber $(x,v^{-}(x))$ to $(x,v^{+}(x))$. Therefore, $\sigma^*$ has  a periodic point in the support of $\mu_\infty^*.$
If the maximizing probability $ \mu_\infty$ is unique for $A$, then  $\mu_\infty^*$
is unique for the maximization problem for $A^*$.  In this case the support of  $\mu_\infty^*$ is this periodic orbit. Therefore, there is a minimal distance (in vertical fiber) between non-graph points and this is in contradiction with the contraction on vertical fibers. The conclusion is that $S$  is a graph if the support of the maximizing probability $\mu_\infty$
is a periodic orbit.

{\bf Remark 4:}
In the case of the shift, if $\supp \mu_{\infty}$ is a periodic orbit, one can easily  show that if  $$\supp \mu_{\infty}=\text{the orbit by $\sigma$ of} \,\,( a_0, a_{1}, ..., a_{(n-1)}, a_{0}, ...)$$  then $$\supp \mu_{\infty}^{*} = \text{orbit by $\sigma^*$ of}\,\,(a_{(n-1)},... , a_2, a_1, a_0, a_{(n-1)},...).$$

\begin{center}
\includegraphics[scale=0.6]{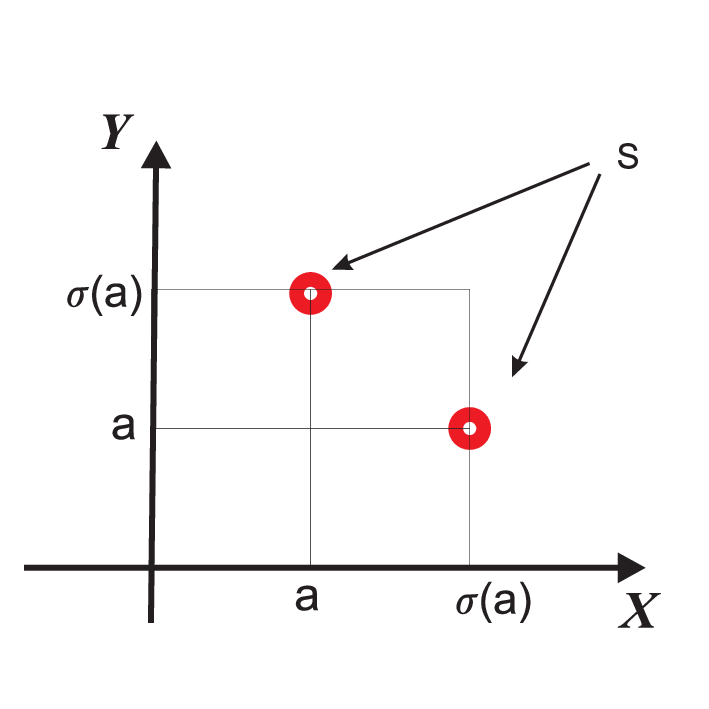}\\
\small{Support of $\hat{\mu}$ in the periodic case.}
\end{center}

We suppose from now on that the support of the maximizing probability $\mu_\infty$
is not a periodic orbit.

\begin{center}
\includegraphics[scale=0.7,angle=0]{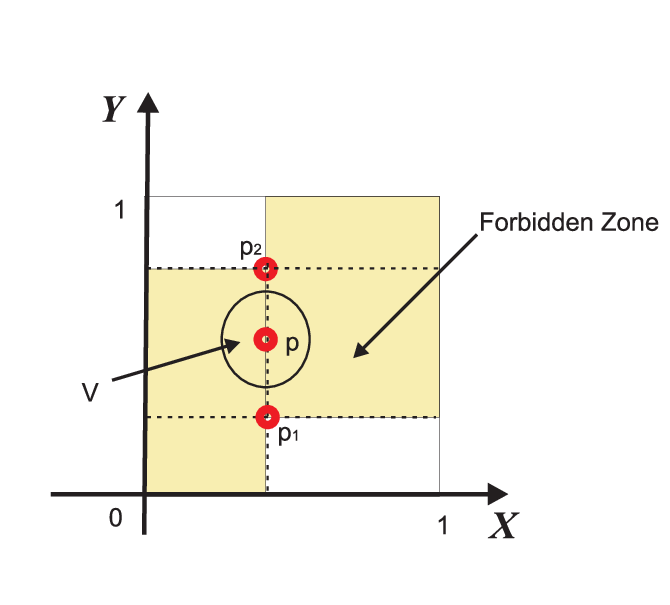}\\
\small{Characterization of $S$}  \\
\end{center}

Suppose, that $v^{-}(x) < v^{+}(x)$ for some $x$, then we claim that  there is no other point in support of $\hat{\mu}$  in the fiber by $x$  between
$p_1=v^{-}(x) $ and $p_2= v^{+}(x)$. Indeed, from the above picture we see that if there exists a point $(x,p)$ in the support of $\hat{\mu}$
such that  $v^{-}(x)= p_1<p< p_2=  v^{+}(x)$, then, as $\hat{\mu}$ is ergodic,  should exist a point $(q_1,q_2)$
in a small neighborhood $V$ of $(x,p)$ such that returns by a forward $n$-iterate by $\hat{\sigma}$ to $V$.

This iterate has to return to the fiber, and  this contradicts the fact that the support of the maximizing probability $\mu_\infty$
is not a periodic orbit.

If the support of $\mu_\infty$ is not a periodic orbit,  then we claim that there does not exist two pairs
$(x_1,y_1), (x_1,z_1)$ and  $(x_2,y_2), (x_2,z_2)$, in the support of $\hat{\mu}$, such that, the orbits by $\sigma$ of $x_1$ and $x_2$  are different.

In order to simplify the argument and notation  we consider bellow $T^*(x)= 2 x$ (mod 1), but we point out the
reasoning apply to any expanding transformation of degree $d$.
Given $y_n$ and $z_n$, $n=1,2$, there exists a rational point of the form $s_n=\frac{q}{2^k}$, with  $0<q < 2^k,$ $q,k\in \mathbb{N}$,
such that $y_n  < s_n < z_n$, $n=1,2$. Consider the $s_n$ determined by the smallest possible value $k$.

The pair of points $ \hat{T}^{-r} (x_n,y_n)$ and $\hat{T}^{-r}  (x_n,z_n)$, $r\geq 0$,  determine non-graph points in the same fiber,
for any $r>0$, until time $r=k$. In time $r=k-1$,  it happens  for the first time that the horizontal fiber through $1/2$ cuts the vertical segment connecting $ \hat{T}^{-(k-1)} (x_n,y_n)$ and $\hat{T}^{-(k-1)} (x_n,z_n)$.

In this way, for each $n$, we get a  horizontal forbidden region $A_n$ (a horizontal strip from one vertical side to the other vertical side of $[0,1]\times [0,1]$) determined by such pair $ \hat{T}^{k-1} (x_n,y_n)$ and $\hat{T}^{-(k-1)} (x_n,z_n)$,  $n=1,2$, which contains  the horizontal fiber through $1/2$ .

If we apply  the argument for $n=1$, then the next forbidden region $A_2$ for $n=2$ will  contain the previous one $A_1$.  Moreover, considering the full forbidden region determined by these  two pair of points we reach a contradiction.

In the picture bellow we show the final pair of points $q_1$ and $q_2$ in a $\hat{\sigma}$-orbit (in the same vertical fiber) which has the property that its images $p_1$ and $p_2$ are on different sides of the upper and down rectangles. The images of $p_1$ and $p_2$ by  $\hat{\sigma}$ are not anymore in the same vertical fiber (neither their future iterates). There is no room for getting a different pair of $p_1$ and $p_2$ like this (because of the forbidden region) .

In this way, form above, we get that could exist just one orbit of $x$ by $\sigma$  such that over the fiber over $x$ there is two points in the support. That is, the projection $K\subset \Sigma$ on the $x$-axis  of the non-graph points have to be the orbit of a single point $x$. Therefore, $\mu_\infty (K) = \sum_k \mu_\infty (\{\sigma^k (x)\}).$

We assume first that the set of non-graph points  have probability $1$ and we will  reach a contradiction.  Indeed, $\mu_\infty (\{\sigma^k (x)\})\geq \mu_\infty (\{\sigma^j (x)\})$, for $k\geq j$,  and  the $\mu_\infty$ probability of the set $\{x\}$ is zero or  is positive.

\begin{center}
\includegraphics[scale=0.8,angle=0]{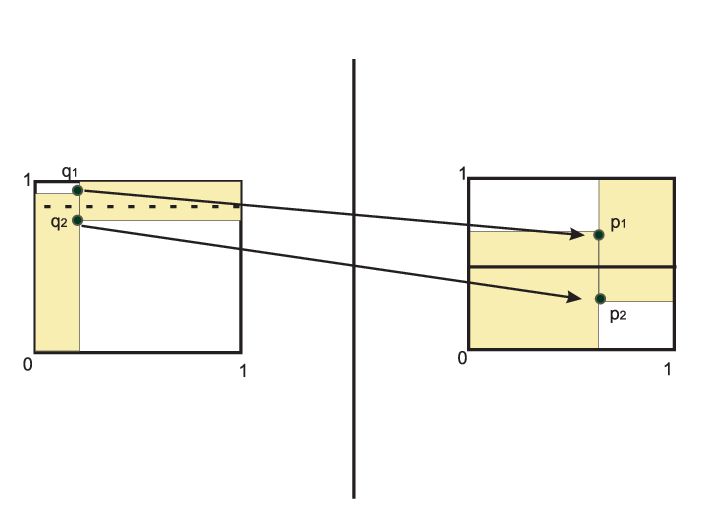}\\
\small{The dynamics on the support}  \\
\end{center}

Remember that the support of $\hat{\mu}$ is invariant by $\hat{\sigma}$.

Now we will show that, indeed, if there exists non-graph points, this set has probability $1$.

Note that if the vertical fiber by $x\in \Sigma$ is such that $v^{-}(x)<v^{+}(x)$, then  $\sigma(x)$ also has this property.
If the transformation $\hat{\sigma}$ we consider preserves orientation in the vertical fiber then the  iterates are in the same order. Otherwise they exchange order. That is, the set of points $(x,y)$ which are not graph point are invariant by forward iteration by $\hat{\sigma}$. Moreover, $\hat{\sigma}$ is a forward contraction in vertical fibers.
Denote by $B=\{\,(x,v^+(x))$ in the support of $\hat{\mu}$ such that $v^{-}(x) < v^+(x)\,\}$.
The set $B$ is the upper part of the non-graph part of the set $S$.

We will show that $\hat{\mu}(B)=0$ or $\hat{\mu}(B)=1$.

We suppose first that $\hat{\sigma}$ preserves  order in the fiber by forward iteration.

Consider $\tilde{B}$ the set $\{\, (x,y)$ in the support of $\hat{\mu}$  such that for some $n\geq 0$ we have  $\hat{\sigma}^n\, (x,y)\in B\,
\}$. Note that as $B$ is forward invariant, once $\hat{\sigma}^n\, (x,y)\in B$, for some fixed $n$, then  $\hat{\sigma}^m\, (x,y)\in B$, for any $m\geq n$.

We will show that  $\hat{\sigma}^{-1} \tilde{B}= \tilde{B}$. The fact that
$\hat{\sigma}^{-1} \tilde{B}\subset  \tilde{B}$ follows easily from the definition of $\tilde{B}$.

Given $x \in \tilde{B}$, there exists $n\geq 0$ such that $\hat{\sigma}^n\, (x,y)\in B$. If $n\geq 1$, then
$\hat{\sigma}^{n-1} \,(\hat{\sigma} (x,y))\,\in B$ and, therefore, $(x,y)\in \hat{\sigma}^{-1} \tilde{B}. $
In the other case $(x,y) \in B$, but then $(\hat{\sigma} (x,y))\in B$, because $\hat{\sigma}$ preserves  order in the fiber, and does not exist  more than two points in the vertical fiber over $\sigma(x)$ which are in $S$.
Therefore, $(x,y) \in \hat{\sigma}^{-1} \tilde{B}$.

As $\hat{\mu}$ is ergodic, then $\hat{\mu} (\tilde{B})=0$ or  $\hat{\mu} (\tilde{B})=1$.

If $\hat{\mu} (\tilde{B})=1$, then take a Birkhoff point $z\in \tilde{B} $  for the ergodic probability $\hat{\mu}$.
Therefore, we get that the asymptotic frequency of visit to the set
$C=\{\,(x,v^-(x))$ in the support of $\hat{\mu}$ such that $v^{-}(x) < v^+(x)\,\}$
(the bellow part of the non-graph part of  set $S$) is zero. Finally, we get that $\hat{\mu} (C)=0$.
In the same way $\hat{\mu} (B)=1$.

If $\hat{\mu} (\tilde{B})=0$, we get that $\hat{\mu} (B)=0$. Now, using a similar argument for the lower part of the non-graph part we get that $\hat{\mu} (C)=1$.

This shows that the $\pi_1$ projection of the non-graph points has probability one and this proves the theorem.

\end{proof}

The above reasoning also applies to $T(x)=-\, 2 x$ (mod 1) and to the shift in the Bernoulli space.

\bigskip
\section{Selection of minimizing sequences}\label{secaogsel}

 In this section we want to exhibit a nice expression for the function $f$ (defined before) such that, the
 set $\{(x, \hat{\partial}_c\, f\,  (x))\,|\, x \in \, \, $ support $ \mu_{\infty}\}$ = support of $\hat{\mu}_{max}$,
 in the case the support of $\hat{\mu}_{max}$ is a periodic orbit. In the end of the section we address briefly the general case.

\begin{definition} We say $c: \hat{\Sigma}=\Sigma\times \Sigma \to \mathbb{R}$ upper semicontinuous satisfies the twist condition on
$ \hat{\Sigma}$, if  (bellow we just consider values of $c$ which
are finite) for any $(a,b)\in \hat{\Sigma}=\Sigma\times \Sigma $
and $(a',b')\in\Sigma\times \Sigma $, with $a'> a$, $b'>b$, we
have
\begin{equation}
c(a,b) + c(a',b')  >  c(a,b') + c(a',b).
\end{equation}
\end{definition}

If $W$ is twist and  $c(x,y) = I(x) - W(x,y) +\gamma $, then $c$
is twist. We assume from now on this property.

\begin{theorem}

Suppose the support of $\hat{\mu}_{max}$ is a periodic orbit.
Choose $(x_0,y_0)$ in such way that $x_0\in \Sigma$ is the smaller
point in the projection and  $y_0\in \hat{\Sigma}$ the smaller on
the fiber over $x_0$. From the above, in this case for any given
$z\in \Sigma$, the $f$ defined before is such that
$$ f(z)= [\,(\, c(z, y_n)- c(x_n,y_n)\,) +$$ $$ (\, c(x_n,
y_{n-1})- c(x_{n-1},y_{n-1})\,)\,+...$$
$$+...+ (\, c(x_3, y_2)- c(x_2,y_2)\,)\,\} \, +$$
 $$ (\, c(x_2, y_1)-
c(x_1,y_1)\,)+ (\, c(x_1, y_0)- c(x_0,y_0)\,) \,\,\,] .,$$

where we use all the possible $x_i$ which are in the support of the maximizing probability for $A$ on the left of $z$, and for each $x_i$ we choose the corresponding $y_i$. In the notation of $f$ above, the last one $(x_n,y_n) = (x_n(z), y_n (z))$ is such that
$(x_n(z), y_n (z))= (x_{k-1}, y_{k-1})$. Which means $n=k-1$.

Moreover,  $x_0<x_1<x_2<...<x_n.$

If $z=x_k$ for some element $x_k$ in the support of $\mu_A$, then, in the notation of $f$ above, if $x_{k-1} < z < x_k$,
then $(x_n,y_n) = (x_n(z), y_n (z))$ is such that $(x_n(x_k), y_n
(x_k))= (x_{k-1}, y_{k-1})$. The case $z=x_k$ is include in the
expression above for $f$. In this case $x_k=x_{n+1}$ following the
above notation. The index of the $x_i$ has no dynamical meaning.
\end{theorem}

{\bf Proof:}

Consider the cost $c(x,y)= I(x) - W(x,y) - \gamma$, and a subset $S \subset X \times Y$ $c$-cyclically monotone. Also, assume that $c$ verifies the twist condition:
If $a<a'$ and $b <b'$ then
$$c(a,b)+ c(a',b') > c(a,b')+ c(a',b).$$

In this way, the definition of $c$ implies that:
$$W(a,b)+ W(a',b') < W(a,b')+ W(a',b).$$

Define $\Delta(x,x',y)=W(x,y)- W(x',y)$, so the twist condition can be restated as:
if $a<a'$, and $b <b'$, then
$$\Delta(a,a',b) < \Delta(a,a',b').$$

Therefore, if we define the map $y \to \Delta(a,a',y)$ we get a increasing map.

Observe that:\\
i) $\Delta(x,x',y)=- \Delta(x',x,y)$\\
ii) $\Delta(x,x,y)=0$\\
iii) $\Delta(x,x',y)+ \Delta(x',x'',y)=\Delta(x,x'',y)$\\

In particular the map, $y \to \Delta(a',a,y)$ is decreasing if $a'>a$.

Given $f:X \to \mathbb{R}$ a $c$-convex function we define the $c$-subderivative of $f$ in $x  \in X$ as being the set:
$$\partial_{c}f(x )= \{ y \in Y | f(z)-f(x ) \leq c(z,y) -c(x ,y), \forall z \in X \}.$$

Using $c(x,y)= I(x) - W(x,y) - \gamma$ we get,
$$\partial_{c}f(x )= \{ y \in Y | f(z)-f(x ) \leq I(z)-I(x) - [W(z,y) -W(x ,y)], \forall z \in X \}.$$

We know that $S$ is $c$-cyclically monotone, if and only if, $S \subset \hat{\partial}_{c}f(x_{0})$ where $f$ is a $c$-convex function given by:

$$f(z)= min_{(x_{i},y_{i}) \subset S, i=1..n} \sum_{i=0}^{n} c(x_{i+1}, y_{i })-c(x_{i }, y_{i })$$
where $(x_{0},y_{0}) \in S$ is as fixed point and  $x_{n+1}=z$.
Using $c(x,y)= I(x) - W(x,y) - \gamma$ we get,
$$f(z)= min_{(x_{i},y_{i}) \subset S, i=1..n} \sum_{i=0}^{n} I(x_{i+1})-I(x_{i}) - [ W(x_{i+1}, y_{i })-W(x_{i }, y_{i })] =$$
$$= min_{(x_{i},y_{i}) \subset S, i=1..n} \sum_{i=0}^{n} I(x_{i+1})-I(x_{i}) + [ \Delta (x_{i }, x_{i+1}, y_{i })]=$$
$$= min_{(x_{i},y_{i}) \subset S, i=1..n} I(z)-I(x_{0}) + \sum_{i=0}^{n}  \Delta (x_{i }, x_{i+1}, y_{i }).$$

\begin{lemma}
If, $(x_{i},y_{i}) \subset S, i=0,1,2$ is such that $x_{0}<x_{1}<x_{2} < z $ and $y_{2}<y_{1}<y_{0}$ then,
$$ \Delta (x_{0 }, x_{1 }, y_{0 }) + \Delta (x_{1 }, z, y_{1 }) > \Delta (x_{0 }, x_{1 }, y_{0 }) + \Delta (x_{1 }, x_{2 }, y_{1 }) + \Delta (x_{2 }, z, y_{2 })$$
\end{lemma}
\begin{proof}

Observe that, $\Delta (x_{1 }, z, y_{1 })=\Delta (x_{1 }, x_{2 }, y_{1 })+ \Delta (x_{2 }, z, y_{1 })  > \Delta (x_{1 }, x_{2 }, y_{1 })+ \Delta (x_{2 }, z , y_{2 })$, because $\Delta (x_{2 }, z, \cdot)$ is increasing and $y_{1}> y_{2}$.

\end{proof}

\begin{lemma}
If, $(x_{i},y_{i}) \subset S, i=0,1,2$ is such that $x_{0}<x_{1}< z < x_{2}$ and $y_{2}<y_{1}<y_{0}$ then,
$$ \Delta (x_{0 }, x_{1 }, y_{0 }) + \Delta (x_{1 }, z, y_{1 }) < \Delta (x_{0 }, x_{1 }, y_{0 }) + \Delta (x_{1 }, x_{2 }, y_{1 }) + \Delta (x_{2 }, z, y_{2 }).$$

\begin{center}
\includegraphics[scale=0.45,angle=0]{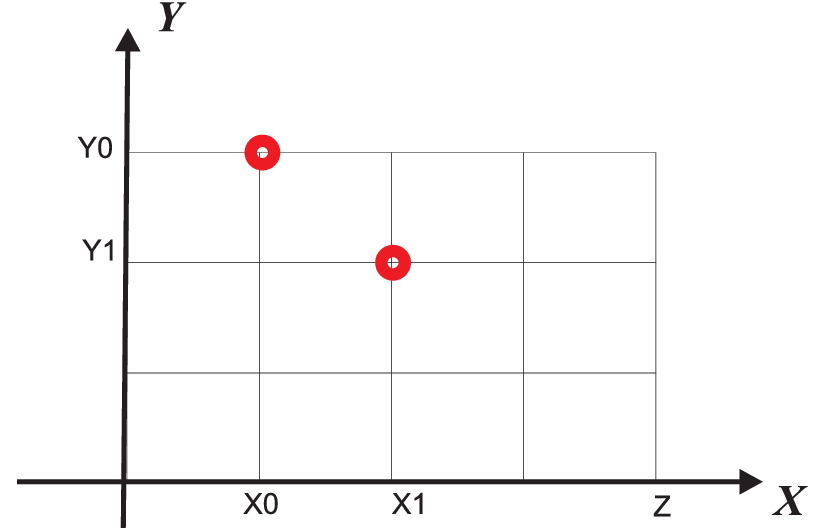}\\
\small{Figure 1 - bad}  \\
\end{center}

\begin{center}
\includegraphics[scale=0.6,angle=0]{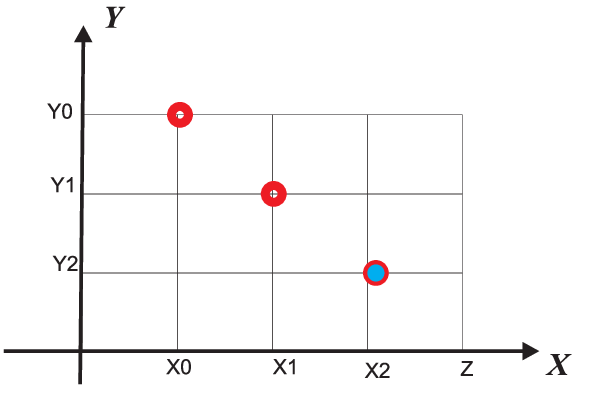}\\
\small{Figure 2 - good}  \\
\end{center}

In particular, $$ \Delta (x_{0 }, x_{1 }, y_{0 }) + \Delta (x_{1 }, z, y_{1 }) < \Delta (x_{0 }, x_{2 }, y_{0 }) + \Delta (x_{2 }, z, y_{2 }).$$
\end{lemma}
\begin{proof}

Observe that, $\Delta (x_{1 }, z, y_{1 })=\Delta (x_{1 }, x_{2 }, y_{1 })+ \Delta (x_{2 }, z, y_{1 })  < \Delta (x_{1 }, x_{2 }, y_{1 })+ \Delta (x_{2 }, z , y_{2 })$, because $\Delta (x_{2 }, z, \cdot)$ is decreasing and $y_{1}> y_{2}$.

\begin{center}
\includegraphics[scale=0.6,angle=0]{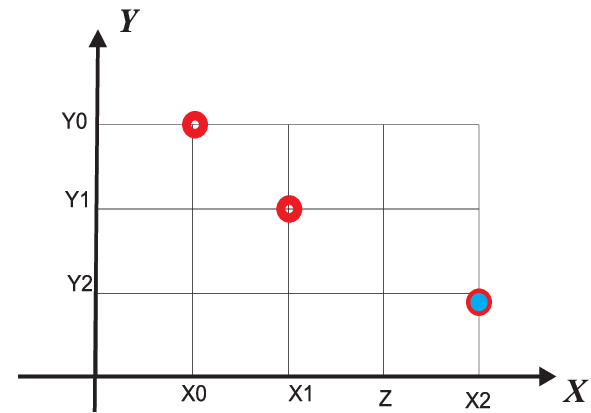}\\
\small{Figure 3 - bad}  \\
\end{center}

\begin{center}
\includegraphics[scale=0.6,angle=0]{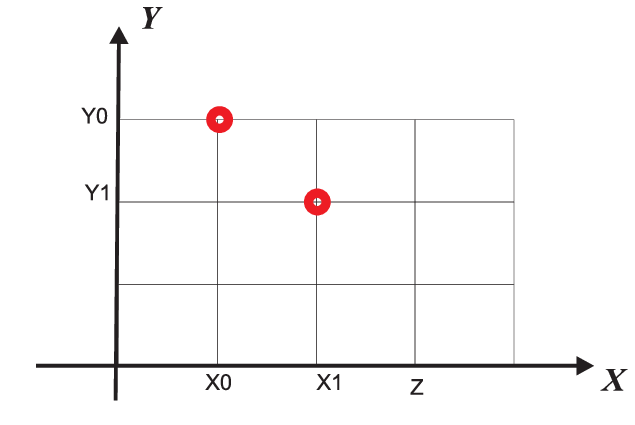}\\
\small{Figure 4 - good}  \\
\end{center}

Now observe that,\\
 $\Delta (x_{0 }, x_{2 }, y_{0 }) + \Delta (x_{2 }, z, y_{2 })=\Delta (x_{0 }, x_{1 }, y_{0 }) +\Delta (x_{1 }, x_{2 }, y_{0 })+ \Delta (x_{2 }, z, y_{2 }) > \\ \Delta (x_{0 }, x_{1 }, y_{0 }) +\Delta (x_{1 }, x_{2 }, y_{1 })+ \Delta (x_{2 }, z, y_{2 }) > \Delta (x_{0 }, x_{1 }, y_{0 }) + \Delta (x_{1 }, z, y_{1 })$.

\end{proof}

Now one can generalize the idea above:
Suppose that, $(x_{i},y_{i}) \subset S, i=0,1,2, ...,n$ is such that $x_{0}<x_{1}< ... <x_{k}< z < x_{k+1} < ...< x_{n}$ and $ y_{n}<...<y_{2}<y_{1}<y_{0}$ then,\\
$\Delta (x_{0 }, x_{1 }, y_{0 }) + \Delta (x_{1 }, x_{2 }, y_{1 }) + ... + \Delta (x_{k }, z, y_{k }) < \\ \Delta (x_{0 }, x_{1 }, y_{0 }) + \Delta (x_{1 }, x_{2 }, y_{1 }) + ... + \Delta (x_{n }, z, y_{n }).$

In order to see this, we proceed by induction in the right side of the inequality above:\\
$\Delta (x_{n-1 }, x_{n }, y_{n-1 }) + \Delta (x_{n }, z, y_{n }) >$ \\
$\Delta (x_{n-1 }, x_{n }, y_{n-1 }) + \Delta (x_{n }, z, y_{n-1 }) =$ \\
$\Delta (x_{n-1 }, z, y_{n-1 })$ \\
In this step we discard the pair $(x_{n }, y_{n })$. We must to repeat this process while $n-j >k$, discarding all points in the right side of $z$.

So the conclusion is, that we can discard all points in the right side of $z$ decreasing the sum, and we can introduce a point between the last point in the left size of z, and z, decreasing the sum (see Figures 3 and 4).
\begin{center}
\includegraphics[scale=0.7,angle=0]{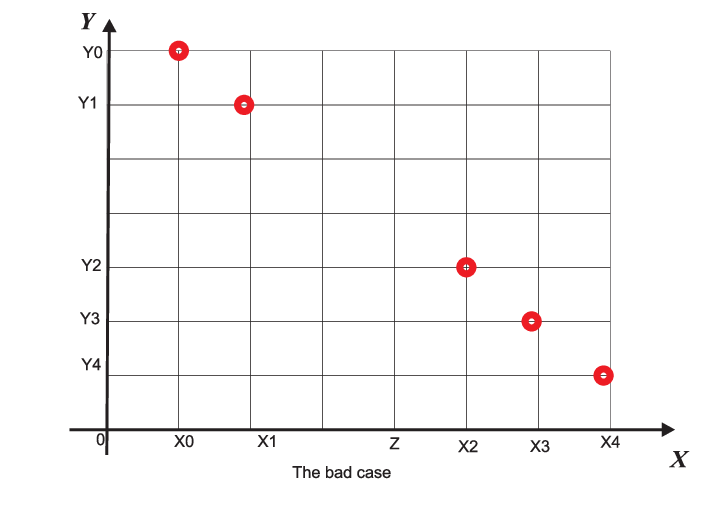}\\
\small{Figure 5}  \\
\end{center}

We discard $(x_2,y_2), (x_3,y_3), (x_4,y_4), $  from right size and insert $(A,B)$ between $(x_1,y_1)$ and $z$.
\begin{center}
\includegraphics[scale=0.6,angle=0]{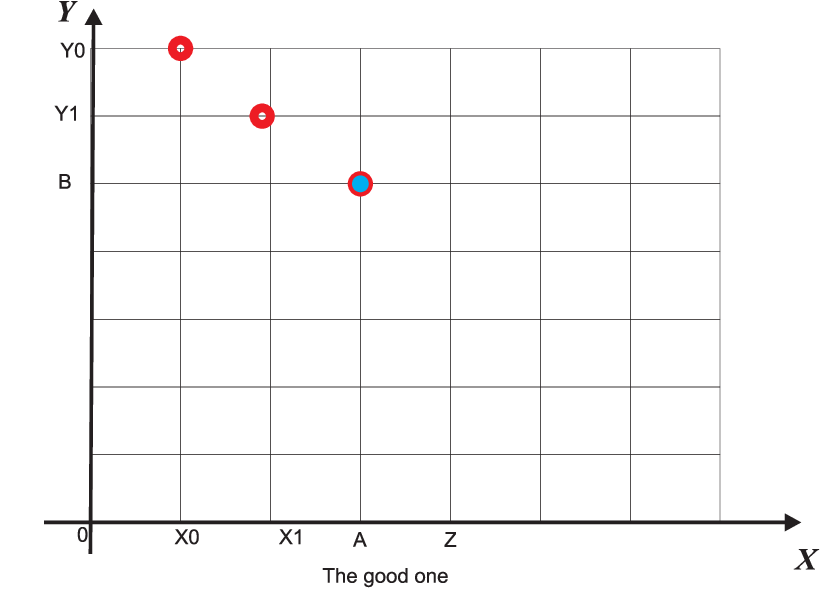}\\
\small{Figure 6}  \\
\end{center}

The case in which $z < x_0$ must be analyzed now:
\begin{center}
\includegraphics[scale=0.6,angle=0]{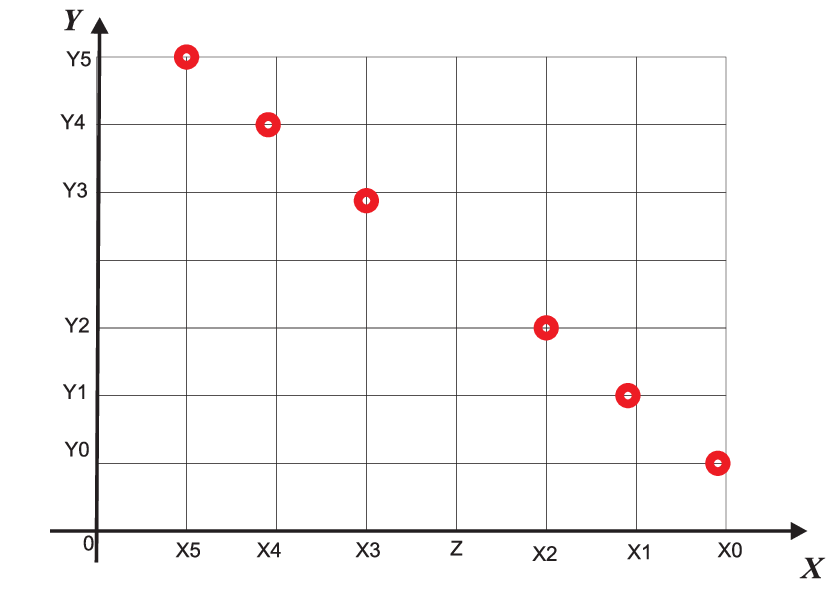}\\
\small{Figure 7 - bad}  \\
\end{center}

\begin{center}
\includegraphics[scale=0.6,angle=0]{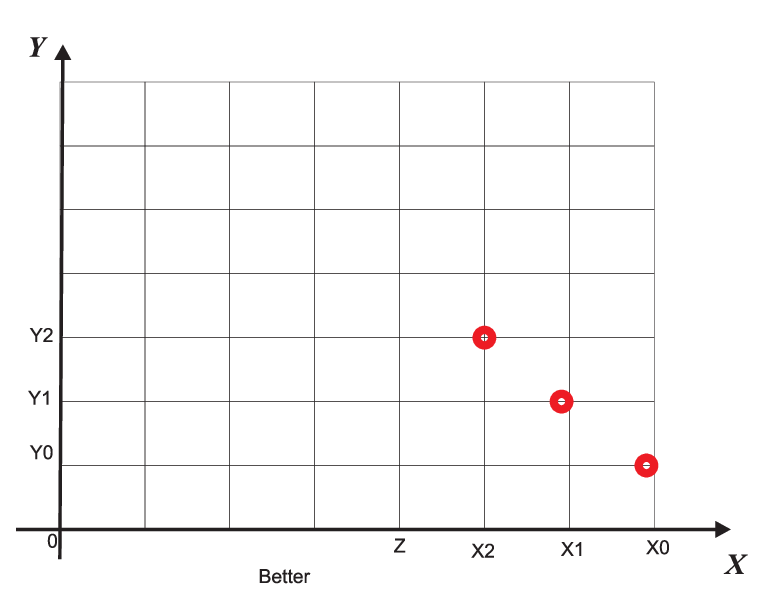}\\
\small{Figure 8 - good}  \\
\end{center}

Observe that:\\
$\Delta (x_{0 }, x_{1 }, y_{0 }) + \Delta (x_{1 }, x_{2 }, y_{1 }) + \Delta (x_{2}, x_{3 }, y_{2}) +\Delta (x_{3 }, x_{4 }, y_{3 }) +\Delta (x_{4 }, x_{5 }, y_{4 }) + \Delta (x_{5 }, z, y_{5 })  > $ \\
$\Delta (x_{0 }, x_{1 }, y_{0 }) + \Delta (x_{1 }, x_{2 }, y_{1 }) + \Delta (x_{2}, x_{3 }, y_{2}) +\Delta (x_{3 }, x_{4 }, y_{3 }) + [ \Delta (x_{4 }, x_{5 }, y_{4 }) + \Delta (x_{5 }, z, y_{4 }) ]  = $ \\
$\Delta (x_{0 }, x_{1 }, y_{0 }) + \Delta (x_{1 }, x_{2 }, y_{1 }) + \Delta (x_{2}, x_{3 }, y_{2}) +\Delta (x_{3 }, x_{4 }, y_{3 }) +  \Delta (x_{4 }, z, y_{4 }), $ \\
and successively to eliminate 4 and 3.

\vspace{0.2cm}
Now we check what happen with permutations of the order in the projected points.
\vspace{0.2cm}

Note that the sum
$$\sum_{i=0}^{n} c(x_{i+1}, y_{i })-c(x_{i }, y_{i })$$
can change by sorting the sequence of points $(x_{i},y_{i}) \subset S, i=1..n$. So we need to consider the natural question about the better way to rename this points.

Please, check the bellow figure:
\begin{center}
\includegraphics[scale=0.7,angle=0]{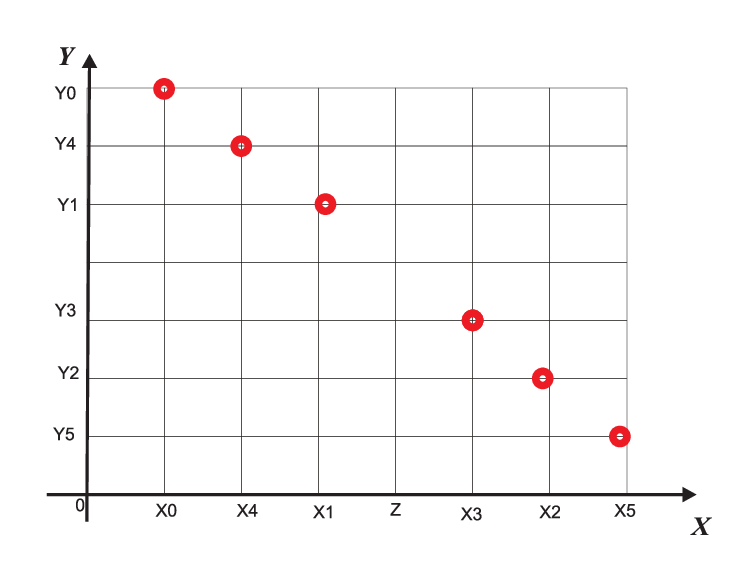}\\
\small{Figure 9 - too bad}  \\
\end{center}

We claim that it is possible discard all the points at the right side of $z$ and also all the points between $x_0$ and $z$ that are no ordered in order to minimize the sum above.

In fact:
$\Delta (x_{0 }, x_{1 }, y_{0 }) + \Delta (x_{1 }, x_{2 }, y_{1 }) + \Delta (x_{2}, x_{3 }, y_{2}) +\Delta (x_{3 }, x_{4 }, y_{3 }) + [ \Delta (x_{4 }, x_{5 }, y_{4 }) + \Delta (x_{5 }, z, y_{5 })]  > $ \\
$\Delta (x_{0 }, x_{1 }, y_{0 }) + \Delta (x_{1 }, x_{2 }, y_{1 }) + \Delta (x_{2}, x_{3 }, y_{2}) + [\Delta (x_{3 }, x_{4 }, y_{3 }) +\Delta (x_{4 }, z, y_{4 })] > $ \\
$\Delta (x_{0 }, x_{1 }, y_{0 }) + [ \Delta (x_{1 }, x_{2 }, y_{1 }) + \Delta (x_{2}, x_{3 }, y_{2})] + \Delta (x_{3 }, z, y_{3 }) ] > $ \\
$\Delta (x_{0 }, x_{1 }, y_{0 }) +[ \Delta (x_{1 }, x_{3 }, y_{1 }) + \Delta (x_{3 }, z, y_{3 }) ] > $ \\
$\Delta (x_{0 }, x_{1 }, y_{0 }) + \Delta (x_{1 }, z, y_{1 }) .$ \\

So the sequence $(x_0,y_0), (x_1,y_1)$ in this order minimize this sum.
\vspace{0.3cm}

We know that the graph property is true. But suppose we have a more general case where $\Delta (x, z, y)$
can be consider and we do not have the graph property.

Consider the sequence  $(x_0,y_0), (x_1,y_1)$ and suppose $z > x_1 >x_0$. Additionally suppose that $(x_1, .) \cap S \neq \{y_1\}$, so we can compares the sum $\Delta (x_{0 }, x_{1 }, y_{0 }) + \Delta (x_{1 }, z, y_{1 })$ with $\Delta (x_{0 }, x_{1 }, y_{0 }) + \Delta (x_{1 }, z, y).$ for any $y \in (x_1, .) \cap S \neq \{y_1\}$.

We claim that this function is monotone increasing in $y$.

\begin{center}
\includegraphics[scale=0.5,angle=0]{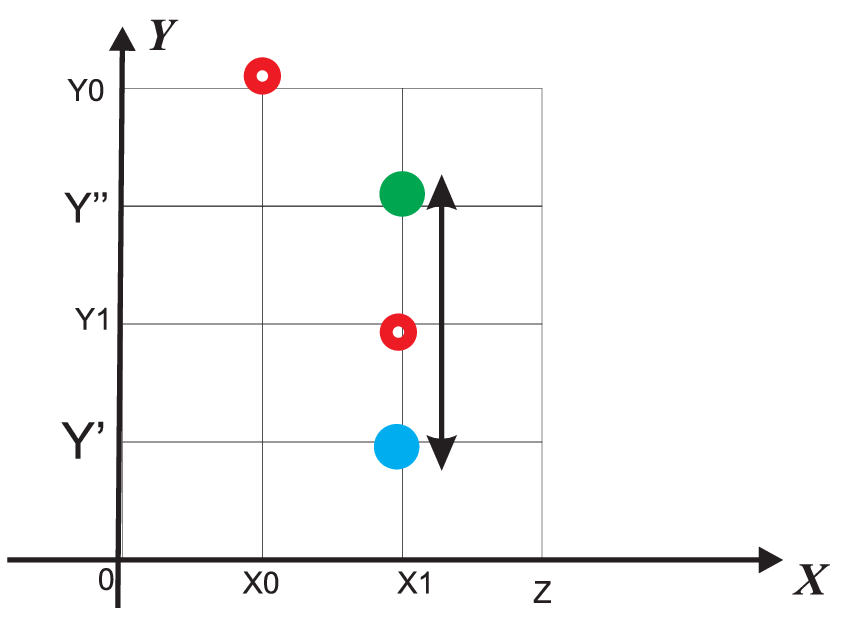}\\
\small{Figure 10 - going down is better}  \\
\end{center}

In fact suppose that $y' < y_1 < y'' < y_0$, as in Fig. 8.

Observe that, $\Delta (x_{1 }, z, y_{1 }) < \Delta (x_{1 }, z, y'')$ and $\Delta (x_{1 }, z, y_{1 }) > \Delta (x_{1 }, z, y')$ because $ x_1 < z$.

The conclusion is that if the support of $\hat{\mu}_{max}$ is a
periodic orbit, then, we choose $(x_0,y_0)$ in the support of
$\hat{\mu}_{max}$.

From the above, in this case given $z\in \Sigma$, then
$$ f(z)= [\,(\, c(z, y_n)- c(x_n,y_n)\,) +$$ $$ (\, c(x_n,
y_{n-1})- c(x_{n-1},y_{n-1})\,)\,+...$$
$$+...+ (\, c(x_3, y_2)- c(x_2,y_2)\,)\,\} \, +$$
 $$ (\, c(x_2, y_1)-
c(x_1,y_1)\,)+ (\, c(x_1, y_0)- c(x_0,y_0)\,) \,\,\,] .,$$

where we use all the possible $x_i$, $i=1,2,..,n,$ on the left of
$z$ , and for each $x_i$ we choose the corresponding $y_i$ such
that $(x_i,y_i)$ is in the support of $\hat{\mu}_{max}$. Moreover,
$x_0<x_1<x_2<...<x_n.$ \vspace{0.3cm}

Finally, we can say that $\hat{\partial}_c  f (x_k)=y_k$, for any
$k$.

One can get similar results for the function $g$
(obtained just from the kernel $W$) defined before.

From the reasoning above (for the case of $W$ satisfying the twist
condition), in the case $\mu_{\infty}$ is not a periodic orbit,
then in definition of $f$, the infimum is not attained in a finite
sequence of $x_n$ in the support of $\mu_\infty$.
\vspace{0.2cm}

\bigskip
\section{Appendix}\label{appendix}

Here we consider first  the shift $\Sigma=\{0,1\}^{\mathbb{N}}$, and $\Sigma$ as a metric space with the usual distance:
$$ d(x,y)=
\begin{cases}
0 & if \quad x=y \\
(1/2)^{n} & if \quad n=\min\{ i | x_{i} \neq y_{i} \}.
\end{cases}
$$
Additionally, we suppose that $\Sigma$ is ordered  by
$x < y$, if $x_i=y_i$ for $i=1..n-1$, and $x_{n}=0$ and $y_{n}=1$.

As the usual, we consider the dynamical system $(\Sigma, \sigma)$ where $\sigma: \Sigma \to \Sigma$ is given by $\sigma(x)=\sigma(x_{1}, x_{2}, x_{3}, ...)=(x_{2}, x_{3}, x_{4}, ...)$.

\vspace{0.2cm}
{\bf a) Potentials and the involution kernel}
\vspace{0.2cm}

As usual we denote
$$\tau_{x}^{*}(y)=(x_{1},y_{1}, y_{2}, y_{3}, ...)\text{ and } \tau_{y}(x)=(y_{1},x_{1}, x_{2}, x_{3}, ...),$$
and
$$\hat{\sigma}(x,y)=(\sigma(x), \tau_{x}^{*}(y)) \text{ and }\hat{\sigma}^{-1}(x,y)=(\tau_{y}x, \sigma^{*}(y)),$$
the skew product map, where $\sigma^{*}(y=(y_{1}, y_{2}, y_{3}, ...))=(y_{2}, y_{3}, y_{4}, ...)$.

We also define $\tau_{k,y} x= (y_{k},y_{k-1}, ...y_2,y_1, x_0,x_1,x_2,...) $, 

where $x=(x_0,x_1,x_2,...),$ $y=(y_1,y_2,y_3,...)$. In a similar way we define $\tau_{k,y}^* x.$

Given a continuous function $A:\Sigma \to \mathbb{R}$, remember
that a continuous function $W : \Sigma  \times \Sigma  \to \mathbb{R}$ is an involution kernel for $A$  if $(W \circ \hat{\sigma}^{-1}  -W  + A\circ\hat{\sigma}^{-1})(x,y)$ does not depends on $x$; In this case the continuous function $A^{*}(y)=(W \circ \hat{\sigma}^{-1}  -W  + A\circ\hat{\sigma}^{-1})(x,y)$ is called the $W$-dual potential of A.

\vspace{0,2cm}

As  in \cite{BLT} we define the cocycle $\Delta_{A}(x,x',y)$, where
\begin{align*} \Delta_{A}(x,x',y) & = \sum_{n\geq1}A\circ \hat{\sigma}^{-n}(x,y)
-A\circ \hat{\sigma}^{-n}(x',y)\\
& = \sum_{n \geq1}A\circ\tau_{n,y}(x)
-A\circ\tau_{n,y}(x'),
\end{align*}
and its dual version $\Delta_{A^{*}}(x,y,y')$, where
\begin{align*} \Delta_{A^{*}}(x,y,y') & = \sum_{n\geq1}A^{*} \circ \hat{\sigma}^{n}(x,y)
-A^{*} \circ \hat{\sigma}^{n}(x,y') \\
& = \sum_{n\geq1}A^{*}\circ\tau_{n,x}^{*}(y)
-A^{*} \circ\tau_{n,x}^*(y').
\end{align*}

Note that:\\
i) $\Delta_{A}(x,x',y)=- \Delta_{A}(x',x,y)$, in particular $\Delta_{A}(x,x,y)=0$,\\
ii) $\Delta_{A}(x,x',y)+ \Delta_{A}(x',x'',y)=\Delta_{A}(x,x'',y)$,\\
iii) $\Delta_{A}(x,x',y)= \Delta_{A}(\tau_{y}x,\tau_{y}x',\sigma^{*}(y)) +[ A \circ\tau_{y}x
-A\circ\tau_{y}x'],$\\
and the same relations are true for $\Delta_{A^{*}}(x,y,y')$.

Using this properties one can prove that, for any involution kernel we have
$W(x,y)-W(x',y)=\Delta_{A}(x,x',y)\text{ and } W(x,y)-W(x,y')=\Delta_{A^{*}}(x,y,y').$

From this fact, we get that the difference between two involution kernels for $A$ is a continuous function of $y$:
$$\{ \text{Involution kernels for} \, A \} / \text{C}^{0}(\Sigma) = W^{0},$$
where $W^{0}(x,y)= \Delta_{A}(x,x',y)$ for a fix $x' \in \Sigma$ is called a fundamental involution kernel of $A$.  Indeed, the property (iii) shows that $W^{0}$ is an involution kernel for $A$.

On the other hand, given another involution kernel, $W$ we have $W(x,y)-W(x',y)=\Delta_{A}(x,x',y)$, thus $$W(x,y) = W(x',y)+ \Delta_{A}(x,x',y)= W(x',y) + W^{0}(x,y)= g(y) + W^{0}(x,y),$$ where $g(y)= W(x',y) \in C^{0}(\Sigma)$.

As an example we compute the general dual potential. First for $W^{0}(x,y)= \Delta_{A}(x,x',y)$ we get:
\begin{align*} A^{*}_{0}(y) & = (W^{0}(\tau_{y}x,\sigma^{*}(y))  - W^{0}(x,y)  + A(\tau_{y}x)\\
&= \Delta_{A}(\tau_{y}x, x',\sigma^{*}(y)) - \Delta_{A}(x,x',y) + A(\tau_{y}x)\\
&= A(\tau_{y}x') + \Delta_{A}(\tau_{y}x',x',\sigma^{*}(y)).\\
\end{align*}
Given another involution kernel, $W$ we have $W(x,y) = W(x',y) + W^{0}(x,y)$ thus
$$A^{*}(y)=(W \circ \hat{\sigma}^{-1}  -W  + A\circ\hat{\sigma}^{-1})(x,y)=W(x',\sigma^{*}(y))- W(x',y) + A^{*}_{0}(y).$$

\vspace{0.2cm}
{\bf b) The twist property of an involution kernel}
\vspace{0.2cm}

If $A:\Sigma \to \mathbb{R}$ is a potential and $W$ an arbitrary involution kernel for $A$, as we said before,  $W$ has the twist property, if for any, $a,b,a',b' \in \Sigma $
$$W(a,b)+ W(a',b') < W(a,b')+ W(a',b),$$
provided that $a<a'$ and $b<b'$.

If we rewrite this inequality as,
\begin{align*}
W(a,b)+ W(a',b') & < W(a,b')+ W(a',b)\\
W(a,b)-W(a',b)   & < W(a,b')-W(a',b') \\
\Delta_{A}(a,a',b) & < \Delta_{A}(a,a',b'),
\end{align*}
we get an alternative criteria for the twist property, that is, $W$ has the twist property, if  for any, $a,a' \in \Sigma $ the function $ y \to \Delta_{A}(a,a',y),$ is strictly increasing, provided that $a<a'$.

\vspace{0.2cm}

{\bf Remark 5}
This characterization shows a very important fact. The twist property is a property of $A$, so we can said that $A$ is a twist potential or equivalently $A$ has a twist involution kernel (as, obviously  other involution kernel  is also twist).
\vspace{0.2cm}

\vspace{0.2cm}
{\bf Remark 6}
As an initial approximation we can consider a different setting of dynamics. Let $T(x)=-2x\text{ mod }1$, and
$$\tau_{0} x = -\frac{1}{2}x + \frac{1}{2} ,\text{ and } \tau_{1} x = -\frac{1}{2}x + 1,$$
the inverse branches that defines the skew maps (that are not the actual natural extension of $T$):
$$\hat{T}(x,y)=(T(x), \tau_{x}^{*}(y)) \text{ and }\hat{T}^{-1}(x,y)=(\tau_{y}x, T^{*}(y)).$$

So, one can compute an involutive (that is, $A^{*}(y)=A(y)$) smooth  kernel  for $A_{1}(x)=x$ and $A_{2}(x)=x^{2}$ given by
$$W_{1}(x,y)=-\frac{1}{3}(x+y)
\text{ and }
 W_{2}(x,y)=\frac{1}{3}(x^{2}+y^{2})-\frac{4}{3}xy .$$

As a corollary we get that any potential $A(x)=a + b x + c x^{2}$ has a   smooth  involution kernel given by $ W(x,y)=a + b W_{1}(x,y) + c W_{2}(x,y). $

Here and in the next paragraphs, we will denote
$$W_{A}(x,y):=a + b W_{1}(x,y) + c W_{2}(x,y),$$
where $A(x)=a + b x + c x^{2}$ is a polynomial  of degree 2.

We observe that the twist property can be derived from the positivity of the second mix derivative of the involution kernel when it is smooth. Note that,
$$\frac{\partial^{2} W_{1}}{\partial x \partial y}=0,\text{ and } \frac{\partial^{2} W_{2}}{\partial x \partial y}=-\frac{4}{3},$$
thus $W_{1}$ is not twist and $W_{2}$ is. Actually any potential $A(x)=a + b x + c x^{2}$ where $c > 0$ is twist.
\vspace{0.3cm}

{\bf Remark 7}\label{NotTwist}
In this remark we are going to consider the case of  $A(x)=a + b x + c x^{2}$ where $c < 0$ (not twist). In this case we will be able to compute the calibrated subaction explicitly, which, we believe, it is interesting in itself.

As a first example consider  $A(x)=-(x-1)^{2}$ which is  a convex potential.

From \cite{JS} \cite{J6} we get that the unique maximizing measure for this potential is $\mu_{\infty}=\delta_{2/3}$, so the critical value is $m=A(2/3)$. Using the fact that that $m=A(2/3)$ one can show that there is a unique (up to constants) calibrated subaction $\phi$ given by:
$$\phi(x)= W(x,2/3) - W(2/3,2/3)=-\frac{1}{3}x^2 + \frac{2}{9}x$$
where the kernel is given by
$$W(x,y)=-(1/3) x^{2}-(1/3)y^{2}+(4/3)xy-(2/3)x-(2/3)y.$$

As a second example consider   $A(x)=-(x-\frac{1}{2})^{2}$ which it is  also  a concave potential.

The general arguments in \cite{JS} shown that any maximizing measure for this potential is $\mu_{\infty}=(1-t)\delta_{1/3}+ t \delta_{2/3}$, where $t \in [0,1]$, so the critical value is $m=A(1/3)=A(2/3)$. In this case the involutive smooth involution kernel is:
$$W(x,y)=-(1/3) x^{2}-(1/3)y^{2}+(4/3)xy-(2/3)x-(1/3)y.$$

It is easy to verify that,
\begin{align*}
\phi(x) &= (W(x,1/3) - W(1/3,1/3)) \chi_{[0,1/2)}(x) +  W(x,2/3) - W(2/3,2/3) \chi_{[1/2,1]}(x)\\
&=\max\{ W(x,1/3) - W(1/3,1/3),  W(x,2/3) - W(2/3,2/3)\}
\end{align*}
is a calibrated subaction for $A$.

In order to prove this, define\\
$V_{1}(x)=W(x,1/3) - W(1/3,1/3)=\Delta(x,1/3,1/3) =-(1/3)x^{2}+(1/9)x$,\\
$V_{2}(x)=W(x,2/3) - W(2/3,2/3)=\Delta(x,2/3,2/3) =-(1/3)x^{2}+(5/9)x-2/9$, and
$$\phi(x)=V_{1}(x) \chi_{_{[0,1/2)}}(x)+V_{2}(x) \chi_{_{[1/2,1]}}(x)=\max\{ V_{1}(x), V_{2}(x)\}.$$

Note that,
\begin{align*}
\phi(\tau_{0}x) &= V_{1}(\tau_{0}x) \chi_{_{[0,1/2)}}(\tau_{0}x)+ V_{2}(\tau_{0}x) \chi_{_{[1/2,1]}}(\tau_{0}x)\\
&=V_{1}(\tau_{0}x)=\Delta(\tau_{0}x,1/3,1/3)\\
&=\Delta(\tau_{1/3}x,\tau_{1/3}1/3,T^{*}1/3)\\
&=\Delta(x,1/3,1/3) - [A(\tau_{1/3}x)-A(\tau_{1/3}1/3)]\\
&=V_{1}(x) - [A(\tau_{0}x)-m].\\
\end{align*}
Thus $\phi(\tau_{0}x)+ A(\tau_{0}x)-m =V_{1}(x)$. Analogously, $\phi(\tau_{1}x)+ A(\tau_{1}x)-m =V_{2}(x)$ so
\begin{align*}
\phi(x) &= \max\{ V_{1}(x), V_{2}(x)\}\\
&= \max\{ \phi(\tau_{0}x)+ A(\tau_{0}x)-m, \phi(\tau_{1}x)+ A(\tau_{1}x)-m\}\\
&= \max_{y \in \Sigma}\{ \phi(\tau_{y}x)+ A(\tau_{y}x)-m\}.\\
\end{align*}

\vspace{0,2cm}

{\bf c) Twist criteria}

\vspace{0,2cm}
Is natural to consider a criteria for the twist property for a class of functions that has a small dependence on the cubic (or higher order) terms. Let $P_{2}^{+}=\{p(x)=a+bx+cx^{2}\, | \,c>0\}$ be the set of strictly convex polynomial. Consider $p \in P_{2}^{+}$, and define $$\mathcal{C}_{\varepsilon}(p)=\{ A \in \text{C}^{3}([0,1]) | A(x)= p(x) + \varepsilon R(x), \frac{\partial}{\partial x}R \in \text{C}^{3}([0,1]) \}$$

\begin{theorem}
For any $p \in P_{2}^{+}$, there exists $\varepsilon>0$ such that all $A \in \mathcal{C}_{\varepsilon}(p)$  is twist.
\end{theorem}
\begin{proof}
Consider $p \in P_{2}^{+}$ fixed. So , $p$ has a smooth and involutive involution kernel given by
$$W_{p}(x,y)=(a+ b W_{1} + c W_{2})(x,y),$$
that is, $p^{*}(y)=p(y)$, where $W_{1}(x,y)=-\frac{1}{3}(x+y)$ and
$W_{2}(x,y)=\frac{1}{3}(x^{2}+y^{2})-\frac{4}{3}xy$, are the involution kernel associated to $x$ and $x^2$ respectively.  Let, $A=p + \varepsilon R \in \mathcal{C}_{\varepsilon}(p)$, and $W_{R}$ be the involution kernel for $R$.  Since $R$ is $\text{C}^{3}$ we get that, is $W_{R}$ is $\text{C}^{2}$ in the variable $x$.

Using the linearity of the cohomological equation, we get
$W_{A}(x,y)$, and differentiating with respect to $x$, we have
$$\frac{\partial}{\partial x}W_{A}(x,y)=(b \frac{\partial}{\partial x} W_{1} + c \frac{\partial}{\partial x} W_{2})(x,y) + \varepsilon \frac{\partial}{\partial x}W_{R}(x,y)=$$
$$
-\frac{1}{3}b + \frac{2}{3} c x - \frac{4}{3} c y + \varepsilon \frac{\partial}{\partial x}W_{R}(x,y)
$$

Since $- \frac{4}{3} c <0$, and $\frac{\partial}{\partial x}W_{R}(x,y) \in \text{C}^{0}([0,1]^{2})$ the compactness of $[0,1]^{2}$ implies that $\frac{\partial}{\partial x}W_{A}(x,\cdot)$ is a strictly decreasing function for any $\varepsilon$ small enough, what is equivalent to the twist property.
\end{proof}

\vspace{0,2cm}
{\bf Remark 9}
If, $A \in \text{C}^{\infty}([0,1])$ is strongly convex, we can consider a perturbation of $A$ of order 2 given by
$$B_{\varepsilon}(x)=A(0)-A'(0) x + \frac{A''(0)}{2} x^2 + \varepsilon \sum_{n\geq 3}\frac{A^{(n)}(0)}{n!} x^n \in \mathcal{C}_{\varepsilon}(p_{A}),$$
where $p_{A}= A(0)-A'(0) x + \frac{A''(0)}{2} x^2 \in P_{2}^{+}$.
Thus, we can find $\varepsilon_{0}>0$ such that $B_{\varepsilon}$ is twist for any $0<\varepsilon< \varepsilon_{0}$.

\vspace{0.6cm}

{\bf d) The involution kernel is bi-Holder}
\medskip

 We consider now $T(x)=2 x$ (mod 1) on the interval $[0,1]$ and the shift $\sigma$ on $\Omega=\{0,1\}^\mathbb{N}$.

A natural question is the regularity of the involution kernel $W$.

We denote $\tau_j$ , $j=0,1$ the two inverse branches of $T$. Given $w =(w_1,w_2,...)\in \{0,1\}^\mathbb{N}$ we denote by $\tau_{k,w}$
the transformation in $[0,1]$ given by
$$ \tau_{k,w}(x)=(\tau_{w_k}\circ \tau_{w_{k-1}}\circ\,...\, \circ\tau_{w_1})\, (x) .$$

We have that, for a fixed $x_0$
$$\Delta(x,x_0,w) = \sum_{k=1}^\infty A( \tau_{k,w} (x)) -  A( \tau_{k,w} (x_0)) $$
and, the involution kernel $W$ can be described as: for any $(x,w)$ we have
$$W(x,w) = \Delta(x,x_0,w).$$

It is easy to see that $W$ is Holder on the variable $x$.

Consider $a,b \in \Omega$ and suppose that $d(a,b)= 2^{-n}.$ In this way $a_j=b_j$, $j=1,2...,n$. We denote $\bar{a}= \sigma^n (a)$ and $\bar{b}= \sigma^n (b)$.

\begin{proposition}
 Suppose $A$ is $\alpha-$Holder. Consider $a,b \in \Omega$ such that $d(a,b)= 2^{-n}.$
For $x \in  [0,1]$ fixed we have
$$|\,W(x,a) - W(x,b)\,|\leq C\, (2^{-n})^\alpha .$$
\end{proposition}

{\bf Proof:}

Note that for $z= \tau_{n,a} (x)= \tau_{n,b} (x)$ and $z_0= \tau_{n,a} (x_0)= \tau_{n,b} (x_0)$ we have
$$ W(x,a) - W(x,b)= \sum_{k=1}^\infty A( \tau_{k,a} (x)) -  A( \tau_{k,a} (x_0)) - A( \tau_{k,b} (x)) +  A( \tau_{k,b} (x_0))= $$
$$ \sum_{k=1}^\infty [\,A( \tau_{k,a} (x)) -   A( \tau_{k,b} (x)) \,] - [\, A( \tau_{k,a} (x_0)) -  A( \tau_{k,b} (x_0))\,]= $$
$$ \sum_{k=1}^\infty [\,A( \tau_{k,\bar{a}} (z)) -   A( \tau_{k,\bar{b}} (z)) \,] - [\, A( \tau_{k,\bar{a}} (z_0)) -  A( \tau_{k,\bar{b}} (z_0))\,].$$

Note also that $|z-z_0|\leq d(a,b)= 2^{-n}.$

Consider $z=z_0 +h$, then
$$ \,A( \tau_{k,\bar{a}} (z_0+h))- A( \tau_{k,\bar{a}} (z_0)) \leq C_A \, d( \tau_{k,\bar{a}} (z_0+h) ,\tau_{k,\bar{a}} (z_0) )^\alpha \leq $$
$$C_A \, (\, 2^{-k} \,h\,)^\alpha= C_A   (\, 2^{-k} \,)^\alpha\, h^\alpha.$$

Then,
$$ \sum_{k=1}^\infty [\,A( \tau_{k,\bar{a}} (z)) - A( \tau_{k,\bar{a}} (z_0)) \,] - [\, A( \tau_{k,\bar{b}} (z)) -   A( \tau_{k,\bar{b}} (z_0)) \,]$$
$$\leq   C_A   \sum_{k=1}^\infty 2\,  (\, 2^{-k} \,)^\alpha\, h^\alpha \leq  C_A   \sum_{k=1}^\infty 2\,  (2^{\alpha})^{-k} \,\, h^\alpha \leq C\, d(a,b)^\alpha. $$

\medskip

From the above we get:

\begin{theorem} If $A: S^1 \to \mathbb{R}$ is Holder then $W:S^1 \times \{0,1\}^\mathbb{N}\to \mathbb{R}$ is bi-Holder.

\end{theorem}

\medskip
{\bf e) The Fenchel-Rockafellar Theorem}

\medskip

Given $f:\mathbb{R}\to\mathbb{R}$ defined on the variable $x$, the Legendre transform of $f$, denoted by $f^*$, is the function on the variable $p$ defined by
$$ f^*(p)= \sup_{x \in \mathbb{R}} \{p\, x -\,f(x)\}.$$
\medskip

\begin{theorem} (Fenchel-Rockafellar) - Suppose $f(x)$  is smooth strictly convex, $f:\mathbb{R}\to\mathbb{R}$, and,  $g(x)$ is smooth strictly concave, $g:\mathbb{R}\to\mathbb{R}$. Denote by $f^*$ and $g^*$ the corresponding Legendre transforms on the variable $p$.

Then,
$$ \inf_{x \in \mathbb{R} } \, \{ \,f(x)\,-\,g(x)\,\}\,= \,\sup_{p \in \mathbb{R} } \, \{ \,g^*(p)\,-\,f^*(p)\,\}$$
\end{theorem}

\begin{center}
\includegraphics[scale=0.9,angle=0]{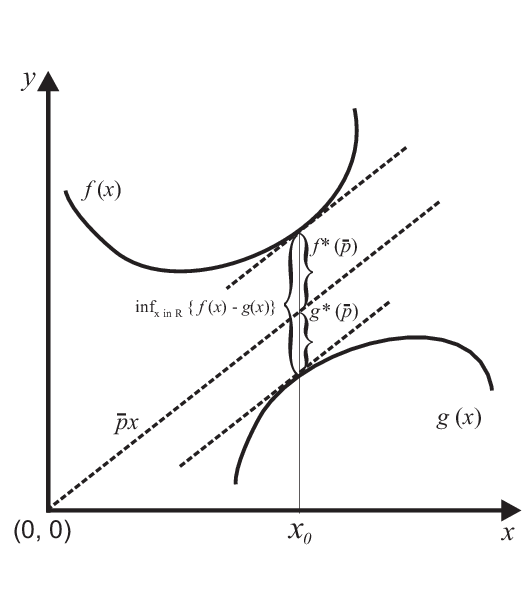}\\
\small{Figure 11}  \\
\end{center}

{\bf Proof:} By convexity and concavity properties we have that there exists $x_0$ such that
$$ \inf_{x \in \mathbb{R} } \, \{ \,f(x)\,-\,g(x)\,\}= f(x_0) - g(x_0).$$

It is also true that $f'(x_0)-g'(x_0)=0$. Denote by $\overline{p}$ that value $\overline{p}=f'(x_0)$.

We illustrate the proof via two pictures in a certain particular case.

Figure 11 shows a geometric picture of the position and values of $f(x_0) - g(x_0)$, $g^*(\overline{p})$ and $f^*(\overline{p}).$
Note that in this picture we have that $f(x_0) - g(x_0)>0.$ This picture also shows the graph of $\overline{p}\,x$ as a function of $x$.

We observe that the Legendre transform is not linear on the function.

Let's consider different values of $p$ and estimate $f^*(p)$ and $g^* (p).$ Suppose first $p\,>\,\overline{p}$.

In figure 12 we show the graph of $p\, x$, and the values of $f^*(p)$ and $g^* (p)$.

We denote by $x_2$ the value such that
$$ f^*(p)= \sup_{x \in \mathbb{R}} \{p\, x -\,f(x)\}\,=\,p\, x_2 - f(x_2).$$
Note that  $x_2\,>\,x_0$.

We denote by $x_1$ the value such that
$$ 0\,<\,g^*(p)= \sup_{x \in \mathbb{R}} \{p\, x -\,g(x)\}\,=\,p\, x_1 - g(x_1).$$
Note that  $x_1\,<\,x_0$.

\begin{center}
\includegraphics[scale=0.9,angle=0]{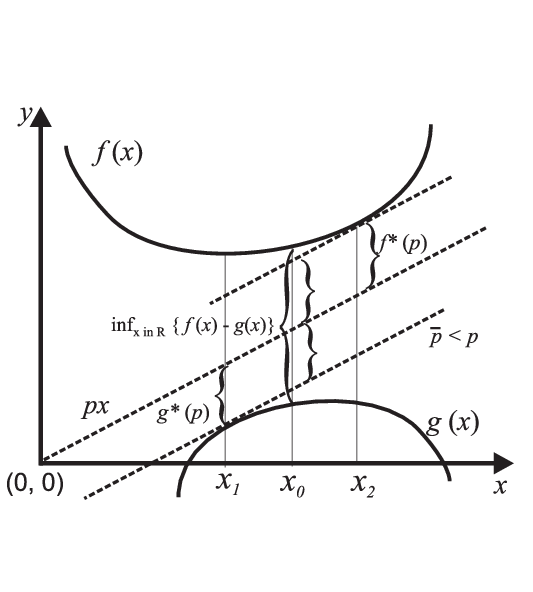}\\
\small{Figure 12}  \\
\end{center}

Note also that $f^*(p)$ and $g^* (p)$ have different signs.

From this picture one can see that $g^* (p) - f^* (p) < f(x_0)- g(x_0).$

In the case  $p\,<\,\overline{p}$ a similar reasoning can be done.

\qed


\begin{thebibliography}{99}







\bibitem [Ban] {Ban} V. Bangert. {\it Mather sets for twist maps and geodesics on tori}, Dynamics Reported 1, 1-56, (1988).

 \bibitem [Ba] {Ba} P. Bhattacharya and M. Majumdar.   {\it  Random Dynamical Systems}. Cambridge Univ. Press, (2007).





 \bibitem [BLT] {BLT} A. Baraviera, A. O. Lopes and Ph. Thieullen.
{\it A Large Deviation Principle for equilibrium states of Holder potentials: the zero temperature case}. Stoch. and  Dyn.(6): 77-96, (2006).

 \bibitem [BLM] {BLM} A.  Baraviera, A. O. Lopes and J. Mengue, On the selection of subaction and  measure for a subclass of potentials defined
 by P. Walters, Erg. Theo. and Dyn. Systems, Volume 33, issue 05,  1338--1362 (2013)


 \bibitem [BLLco] {BLLco}  A.  Baraviera, R. Leplaideur and A. O. Lopes, Ergodic Optimization, Zero Temperature Limits and the Max-Plus Algebra,
  mini-course in XXIX Coloquio Brasileiro de Matem\'atica (2013)


 \bibitem [BOR] {BOR}  A. Baraviera, E. Oliveira, F. B. Rodrigues,
On the dynamics and entropy of the push-forward map, preprint Arxiv (2013)


\bibitem[BCLMS] {BCLMS}
A. T.  Baraviera, L. M. Cioletti,  A. O. Lopes,
J. Mohr  and R. R. Souza,
{\it On the general $XY$ Model: positive and zero temperature, selection and non-selection}, Reviews in Math. Physics. Vol. 23, N. 10,  1063–-1113 (2011).

\bibitem[BG] {BG}
R. Bissacot and E. Garibaldi, {\it Weak KAM methods and ergodic optimal problems for countable Markov shifts}. Bull. Braz. Math. Soc. (N.S.) 41, no. 3, 321–-338 (2010)


\bibitem[B1] {B1}
T. Bousch. {\it Le poisson n'a pas d'ar\^etes}, Annales de
l'Institut Henri Poincar\'e, Probabilit\'es et Statistiques,
Vol 36, 489-508 (2000).


\bibitem[B2]        {B2}
T. Bousch. {\it La condition de Walters}, Annales Scientifiques de
l'\'Ecole Normale Sup\'erieure \textbf{34}, 287-311.  (2001)




\bibitem[CI] {CI}
G. Contreras and R. Iturriaga. {\it Global minimizers of autonomous
Lagrangians}, 22$^\circ$ Co\-l\'o\-quio Brasileiro de Matem\'atica,
IMPA, (1999).



\bibitem[CLT] {CLT}
G. Contreras, A. O. Lopes and Ph. Thieullen. {\it Lyapunov minimizing measures for
expanding maps of the circle}, Ergodic Theory and Dynamical Systems Vol 21, 1379-1409 (2001).



\bibitem[CLO] {CLO} G. Contreras, A. O. Lopes and E. Oliveira,
Ergodic Transport Theory, periodic maximizing probabilities  and  the twist condition,
"Modeling, Optimization, Dynamics and Bioeconomy", Springer Proceedings in Mathematics, Edit. David Zilberman and Alberto Pinto, 183-219, (2014).


\bibitem[CO] {CO} G. Contreras, Ground states are generically a periodic orbit, Arxiv (2013)


\bibitem [CG]     {CG}
J. P. Conze, Y. Guivarc'h. {\it Croissance des sommes ergodiques et
principe variationnel}, manuscript circa 1993.


\bibitem [Del] {Del}
J. Delon, J. Salomon and A. Sobolevski,
{\it Fast transport optimization for Monge costs on the circle},
SIAM J. Appl. Math, no. 7, 2239–-2258, (2010).




 \bibitem [DZ] {DZ} A. Dembo and O. Zeitouni. {\it  Large Deviations Techniques and Applications}. Springer Verlag, (1998).


\bibitem[EG] {EG}   L. Evans and D. Gomes.      {\it Linear Programming interpretation of Mather's
Variational Principle}. ESAIM Control Optim. Cal. Var.,
V. 8, 693-702, (2002)


\bibitem [GP] {GP}
S, Galatolo and M. Pacifico, Lorenz-like flows: exponential decay of correlations for the Poincar\'e map, logarithm law, quantitative recurrence. Ergodic Theory Dynam. Systems 30, no. 6, 1703–-1737 (2010).




\bibitem[GM] {GM} W. Gangbo and R. J. McCann. {\it The Geometry of Optimal
Transportation}. Acta Math, V. 177, 113-161 (1996)


 \bibitem [GL] {GL} E. Garibaldi and A. O. Lopes. {\it On Aubry-Mather theory for symbolic Dynamics},
Ergodic Theory and Dynamical Systems, Vol 28 , Issue 3, 791-815 (2008)





\bibitem[GL1] {GL1}
E. Garibaldi and  A. O. Lopes. {\it Functions for relative maximization}, Dynamical Systems,
V. 22, 511-528, (2007).



\bibitem[GLT] {GLT}
E. Garibaldi,  A. O. Lopes and Ph. Thieullen , On calibrated and separating sub-actions, \emph{Bull. Braz. Math. Soc.}, Vol 40 (4), 577-602, (2009)

\bibitem[GT1] {GT1}
E. Garibaldi and Ph. Thieullen,  {\it Minimizing orbits in the discrete Aubry-Mather model}, Nonlinearity 24, no. 2, 563–-611 (2011)


\bibitem[GT2] {GT2} E. Garibaldi and Ph. Thieullen,
{\it  Description of some ground states by Puiseux technics}, Jour. of Statist. 146, no. 1, 125–180 (2012).



\bibitem[GL4] {GL4} E. Garibaldi and A. O. Lopes,
The effective potential and transshipment in thermodynamic formalism at temperature zero, Stoch. and Dyn., Vol 13 - N 1, 1250009 (13 pages) (2013)









 \bibitem [Go] {Go} C. Gole.   {\it  Sympletic super-twist maps}. World Sci. Pub Co Inc, (1998).




\bibitem[HY] {HY}
B. R. Hunt, G. C. Yuan. {\it Optimal orbits of hyperbolic systems},
Nonlinearity, V. 12, 1207-1224, (1999)

\bibitem[J1] {J1} O. Jenkinson. {\it Ergodic optimization}, Discrete and Continuous
Dynamical Systems, Series A, V. 15, 197-224, (2006)


\bibitem[J2] {J2}
O. Jenkinson, {\it Every ergodic measure is uniquely maximizing},
Discrete and Continuous Dynamical Systems, Series A,
V. 16, 383-392, (2006)

\bibitem[J3] {J3}
O. Jenkinson, {\it A partial order on x2 -invariant measures},
Math. Res. Lett. 15, no. 5, 893–-900 (2008).


\bibitem[JS] {JS} O. Jenkinson and J. Steel
{\it Majorization of invariant measures for
orientation-reversing maps}. Ergodic Theory Dynam. Systems 30, no. 5, 1471–-1483 (2010).

\bibitem[J6] {J6} O. Jenkinson,
{\it Optimization and majorization of invariant measures}, Electron. Res. Announc. Amer. Math. Soc. 13, 1–12 (2007).



\bibitem[Kl] {Kl} B. Kloeckner, Optimal Transport and dynamics of circle expanding maps acting on measures,
Ergodic Theory Dynam. Systems 33, no. 2, 529–548 (2013).

\bibitem[KLS] {KLS} B. Kloeckner, A. O. Lopes and M. Stadlbauer,
Contraction in the Wasserstein metric for some Markov chains, and applications to the dynamics of expanding maps, preprint (2014)

\bibitem[KGLM] {KGLM} B. Kloeckner, P. Giulietti, A. O. Lopes and D. Marcon,
Continuous time transport for Lipchitz  Gibbs probabilities, preprint (2014)


\bibitem[Le] {Le} R. Leplaideur,
{\it A dynamical proof for the convergence of Gibbs measures at temperature zero}. Nonlinearity 18, no. 6, 2847-2880
(2005)



\bibitem[LOS] {LOS} A. O. Lopes, E. R. Oliveira and D. Smania, Ergodic Transport Theory and
Piecewise Analytic Subactions for Analytic Dynamics, Bull. of the Braz. Math Soc. Vol 43 (3) 467-512 (2012)


\bibitem[LM4] {LM4} A. Lopes and J. Mengue, Duality Theorems in Ergodic Transport,
Journal of Statistical Physics. Vol 149, issue 5,  921–-942 (2012)




\bibitem[LMMS] {LMMS}
A. O. Lopes, J. K. Mengue, J. Mohr and  R. R. Souza, Entropy, Pressure and Duality for Gibbs plans in  Ergodic Transport, to appear in Bull. of the Braz. Math, Soc.


\bibitem[LMMS1] {LMMS1}
A. Lopes, J. K. Mengue, J. Mohr and R. R. Souza, Entropy and Variational Principle for one-dimensional Lattice Systems with a general a-priori probability: positive and zero temperature, to appear in Erg. Theo. and Dyn. Syst.


\bibitem[LO] {LO} A. O. Lopes and E. R. Oliveira, On the thin boundary of the fat attractor, preprint UFRGS (2011)

\bibitem[LT1] {LT1}
A. O. Lopes and Ph. Thieullen. {\it Sub-actions for Anosov diffeomorfisms},
Ast\'erisque, V. 287, 135-146, (2003)



\bibitem[LT2] {LT2}
A. O. Lopes  and P. Thieullen. {\it Mather measures and the Bowen-Series
transformation}, Annales de l'Institut Henri Poincar\'e,
Analyse non Lin\'eaire, V. 23, 663-682, (2006)



\bibitem[LT3] {LT3}
A. O. Lopes  and P. Thieullen. {\it Sub-actions for Anosov flows},  Erg Theo and Dyn Systems, Vol 25,  605-628
(2005)

\bibitem[Mat] {Mat}
J. Mather. {\it Action minimizing invariant measures for positive
definite Lagrangian Systems}, Math. Z., 207 (2),  169-207, (1991)



\bibitem[Mi] {Mi} T. Mitra, Introduction to Dynamic Optimization Theory, \emph{Optimization and Chaos}, Editors: M.  Majumdar, T. Mitra and K. Nishimura, Studies in Economic Theory,  Springer Verlag (2000)


\bibitem[Mo] {Mo}
I. D. Morris.  {\it A sufficient condition for the subordination principle in ergodic optimization}, Bull. Lond. Math. Soc. 39, no. 2,  214-220, (2007)


\bibitem[OM] {OM} J. K. Mengue and E. R. Oliveira,
Duality results for Iterated Function Systems with a general family of branches, preprint Arxiv  (2014)


\bibitem[PP] {PP}
W. Parry and M.  Pollicott. {\it Zeta functions and the periodic
orbit structure of hyperbolic dynamics}, Ast\'erisque
Vol {187-188} (1990).




 \bibitem [Ra] {Ra} S. Rachev and L. Ruschendorf.   {\it  Mass transportation problems, Vol I and II}. Springer Verlag, (1998).




\bibitem[R] {R} L. Ruscheendorf. {\it On $c$-optimal random variables},
Statistics and Probability Letters, V. 27, 267-270, (1996)

\bibitem[Sa] {Sa}
S. V. Savchenko, {\it Cohomological inequalities for finite Markov chains},
Functional Analysis and Its Applications 33, 236-238, (1999)


\bibitem[Sou] {Sou} R. R. Souza, {\it Ergodic and Thermodynamic Games}, preprint (2014)

\bibitem[TZ] {TZ} F. A. Tal and S. A. Zanata.
{\it Maximizing measures for endomorphisms of the circle}, Nonlinearity, 21, (2008)




\bibitem[Vi1] {Vi1}
C. Villani, \emph{Topics in optimal transportation}, AMS, Providence (2003)

\bibitem[Vi2] {Vi2}
C. Villani, \emph{Optimal transport: old and new}, Springer-Verlag, Berlin, (2009)

\end{thebibliography}
\end{document}